\documentclass[a4paper,11pt]{article}
\usepackage[T1]{fontenc}
\usepackage[latin1]{inputenc}

\usepackage{amsmath,amsfonts}
\usepackage{amsthm}
\usepackage{amssymb, amsopn}
\usepackage[USenglish]{babel}
\usepackage{amsfonts}
\usepackage{enumerate}
\usepackage{verbatim}
\usepackage{graphicx}
\usepackage{verbatim}
\usepackage[left=3cm, right=3cm, top=3cm, bottom=3cm]{geometry}
\usepackage{tikz}

\usepackage{csquotes}
\usepackage{mathscinet} 

\usepackage[final, 
 pdftex, 
 pdfpagelabels,
 pdfstartview = {FitH}, 
 bookmarks, 
 colorlinks, 
 plainpages = false, 
 linktoc=all, 
 linkcolor=black, 
 citecolor=blue, 
 urlcolor=blue, 
 filecolor=black 
 ]{hyperref}
\usepackage{cleveref}

\usetikzlibrary{matrix,arrows,decorations.pathmorphing,decorations.markings}

\theoremstyle{plain}
\newtheorem{theorem}{Theorem}[section]
\newtheorem{lemma}[theorem]{Lemma}
\newtheorem{prop}[theorem]{Proposition}
\newtheorem{cor}[theorem]{Corollary}

\newtheorem{introthm}{Theorem}

\theoremstyle{definition}
\newtheorem{definition}[theorem]{Definition}

\newtheorem{remark}[theorem]{Remark}

\theoremstyle{remark}

\renewcommand{\tilde}{\widetilde}
\renewcommand{\bar}{\overline}
\newcommand{\bbC}{\mathbb{C}}
\newcommand{\bbD}{\mathbb{D}}

\newcommand{\bbN}{\mathbb{N}}

\newcommand{\bbR}{\mathbb{R}}
\newcommand{\bbZ}{\mathbb{Z}}
\newcommand{\bbLL}{\mathbb{LL}}

\newcommand{\raw}{\rightarrow}

\newcommand{\xra}{\xrightarrow}
\newcommand{\cug}{\subseteq}

\newcommand{\eps}{\varepsilon}
\newcommand{\xto}{\xrightarrow}

\newcommand{\undx}{{\underline{x}}}

\newcommand{\undw}{{\underline{w}}}
\newcommand{\undv}{{\underline{v}}}
\newcommand{\unde}{{\underline{e}}}
\newcommand{\undf}{{\underline{f}}}

\newcommand{\calC}{\mathcal{C}}

\newcommand{\calH}{\mathcal{H}}

\newcommand{\calK}{\mathcal{K}}

\newcommand{\calO}{\mathcal{O}}
\newcommand{\calP}{\mathcal{P}}

\newcommand{\calFl}{\mathcal{F}l}

\newcommand{\flipLL}{\raisebox{\depth}{\scalebox{1}[-1]{$\LL$}}}
\newcommand{\eid}{{id}}
\newcommand{\graff}{\mathcal{G}r}
\newcommand{\flaff}{\widehat{\calFl}}

\newcommand{\sbim}{{\mathbb{S}Bim}}

\newcommand{\ehat}{{\unde(\hat{\imath})}}

\newcommand{\frN}{\mathfrak{N}}

\newcommand{\barw}{{\bar{w}}}

\newcommand{\bfH}{\mathbf{H}}

\newcommand{\Address}{
 \bigskip{\footnotesize

 \textsc{Università di Pisa, Italy}\par\nopagebreak
 \textit{E-mail address}: \texttt{leonardo.patimo@unipi.it}
}}

\DeclareMathOperator{\defect}{def}

\DeclareMathOperator{\End}{End}

\DeclareMathOperator{\Sym}{Sym}

\DeclareMathOperator{\grdim}{grdim}
\DeclareMathOperator{\grrk}{grrk}

\DeclareMathOperator{\Hom}{Hom}
\DeclareMathOperator{\Iden}{Id}

\DeclareMathOperator{\Ker}{Ker}

\DeclareMathOperator{\spa}{span}
\DeclareMathOperator{\IH}{IH}

\DeclareMathOperator{\Trace}{Tr}
\DeclareMathOperator{\LL}{LL}
\DeclareMathOperator{\can}{can}
\DeclareMathOperator{\BS}{BS}
\DeclareMathOperator{\Coh}{H}
\newcommand{\tilLL}{\tilde{\LL}}\newcommand{\tilll}{\tilde{ll}}

\title{A Hom formula for Soergel Modules}

\author{Leonardo Patimo}

\begin{document}

\maketitle

\begin{abstract}
We study Soergel modules for arbitrary Coxeter groups. For infinite Coxeter groups, we show that the homomorphisms between Soergel modules are in general more than those coming from morphisms of Soergel bimodules. This result provides a negative answer to a question posed by Soergel in \cite{SoeKazhdana}.

We further show that the dimensions of the morphism spaces agree with the pairing in the Hecke algebra when Soergel modules are instead regarded as modules over the structure algebra. Moreover, we use this module structure to define a distinguished submodule of indecomposable Soergel bimodules that mimics the cohomology submodule of the intersection cohomology. Combined with the Hodge theory of Soergel bimodules, this can be used to extend results regarding the shape of Bruhat intervals, such as top-heaviness, to arbitrary Coxeter groups.
\end{abstract}

\section*{Introduction}

Let $(W,S)$ be a Coxeter system and $V$ be a representation of $W$. The category of Soergel bimodules, denoted $\sbim$, is the full subcategory of graded bimodules over the polynomial ring $R=\Sym(V)$, generated by direct summands of shifts of Bott--Samelson bimodules
$$\BS(\underline{s_1s_2\ldots s_k}):=R\otimes_{R^{s_1}}R\otimes_{R^{s_2}}R\otimes \ldots \otimes_{R^{s_k}}R(k)$$
where $s_i\in S$, $R^{s_i}$ denotes the subring of $s_i$-invariants and $(k)$ denotes a grading shift. Under certain assumptions on the representation $V$ (cf. \cite{SoeKazhdana,AbeBimodule}), indecomposable self-dual Soergel bimodules are parametrized by elements $w\in W$ and are denoted by $B_w$.

If $W$ is a Weyl group and $V$ is the geometric realization of $W$ over a field $\Bbbk$, then $B_w$ is isomorphic to $\IH_T^\bullet(X_w, \Bbbk)$, the torus equivariant intersection cohomology of the associated Schubert variety $X_w$ with coefficients in $\Bbbk$. 
While the theory of Soergel bimodules can be developed for any Coxeter group, in the general case there is often no known underlying geometric object. Nevertheless, in many aspects these bimodules still behave as if they were the intersection cohomology of some varieties. 

When $\Bbbk=\bbR$, key results from Hodge theory continue to hold for Soergel bimodules. This includes fundamental theorems such as the hard Lefschetz theorem and the Hodge--Riemann bilinear relations \cite{EWHodge}, which hold even in the setting of singular Soergel bimodules \cite{PatSingular}. Moreover, one can carry out further constructions inspired by Hodge theory, such as defining the analogue of the Néron--Severi Lie algebra, in this purely algebraic setting \cite{PatNeron}.

In this paper we describe another aspect in which $B_w$ behaves like $\IH_T^\bullet(X_w, \Bbbk)$. Just as $\IH_T^\bullet(X_w, \Bbbk)$ contains the ordinary $T$-equivariant cohomology $\Coh_T^\bullet(X_w, \Bbbk)$ as a submodule, we show that $B_w$ contains a distinguished submodule $H_w\subset B_w$ playing an analogous role. The submodule $H_w$ is graded free as a left $R$-module, and in degree $2d$ it has a basis indexed by elements in $W$ of length $d$.

We give two equivalent constructions of this submodule.
The first approach uses Fiebig's framework \cite{FieCombinatorics} which realizes Soergel bimodules as certain sheaves on the moment graph. In this setting, we realize $H_w$ as the cyclic module over the structure algebra $Z$ of the Bruhat graph of $W$ generated by the identity.

The second approach uses Libedinsky's light leaves \cite{LibSur}. Light leaves form bases of the homomorphism spaces between Bott--Samelson bimodules, and applying them to $1^\otimes:=1\otimes 1\otimes \ldots\otimes 1$ we also obtain a basis of the bimodule $\BS(\undw)$ itself. Light leaves for $\undw$ are parametrized by sequences $\unde\in \{0,1\}^{\ell(\undw)}$. The construction of the light leaves depends then on a decoration of this sequence. For each index $i$ we decorate the light leaf with a label $U$ we have $s_1^{e_1}s_2^{e_2}\ldots s_{i-1}^{e_{i-1}}s_i>s_1^{e_1}s_2^{e_2}\ldots s_{i-1}^{e_{i-1}}$ and with a label $D$ otherwise. 
 In this setting, we realize $H_w$ as the orthogonal, with respect to the intersection form of $\BS(\undw)$, of all non-canonical light leaves, i.e. those light leaves whose decoration contains at least one $D$.

We show in \Cref{Z1=H} that these two constructions of the bimodule $H_w$ coincide. Moreover, the submodule $H_w$ is naturally endowed with a distinguished basis, which corresponds geometrically to the Schubert basis, i.e. the basis given by the fundamental classes of smaller Schubert varieties.

\begin{introthm}
The indecomposable Soergel bimodule $B_w$ has a submodule $H_w$ which is free as a left (or right) $R$-module with basis $\{\mathcal{P}_x\}_{x\leq w}$. In particular, its graded rank over $R$ is
\[\grrk H_w=v^{-\ell(w)}\sum_{x\leq w}v^{2\ell(x)}.\]
\end{introthm}

Thus, for any Coxeter groups the dimensions of the graded components of $H_w$, which are given by number of elements in $W$ of a certain length, behave like the Betti numbers of a projective variety. This immediately implies that the sequence of these numbers satisfies top-heaviness (extending \cite[Theorem A and C]{BEShape} for crystallographic Coxeter groups). Together with the results in \cite{PatNeron}, this also gives a Hodge-theoretic proof of Carrell--Peterson's criterion for the triviality of Kazhdan--Lusztig polynomials (\cite[Theorem D]{BEShape}).

As the structure algebra $Z$ and the modules $H_w$ are free as left $R$-modules, the quotients $\bar{Z}:=\Bbbk\otimes_R Z$ and $\bar{H_w}:=\Bbbk\otimes_R H_w$ also have Schubert bases. Any Soergel module $\bar{B}_w:=\Bbbk\otimes_R B_w$ is naturally a module over $\bar{Z}$. We claim that this is the ``correct'' module structure one should equip $\bar{B}_w$ with. In fact, using the Schubert basis we can show the module $\bar{B}_w$ remains indecomposable over $\bar{Z}$ and we are able to compute the spaces of homomorphisms:
\begin{introthm}[Soergel's hom formula for Soergel modules]\label{iB}
	Let $B,B'$ be Soergel bimodules. 
	Let $\calH(W,S)$ denote the Hekce algebra of $W$ and let $[-]:[\sbim] \xra{\sim}\calH(W,S)$ be the isomorphism offered by Soergel's categorification theorem (see \Cref{sct} and \Cref{Hiso} for details).
	Then
	\begin{equation}\label{hfi1}
		\Bbbk\otimes_R \Hom_{\sbim}^{\bullet}(B,B')\cong \Hom^{\bullet}_{\bar{Z}}(\bar{B},\bar{B'})
	\end{equation}
	and 
	\begin{equation}\label{hfi2}
		\grdim_{\Bbbk} \Hom^{\bullet}_{\bar{Z}}(\bar{B},\bar{B'})=(\bar{[B]},[B'])
	\end{equation}
$(-,-)$ is the standard pairing in the Hecke algebra (cf. \cite[Definition 3.13]{EMTWIntroduction}).
\end{introthm}
We remark that the isomorphism  \eqref{hfi1} and the formula \eqref{hfi2} do not hold when $\Hom_{\bar{Z}}$ is replaced by $\Hom_{-R}$ (i.e., when $\bar{B}$ and $\bar{B'}$ are only considered with their right $R$-module structure), at least when $W$ is infinite. 
We provide two counterexamples for this. 

For $W$ the affine Weyl group of type $\tilde{A}_2$, we exhibit an indecomposable Soergel bimodule $B_w$ such that  $\bar{B}_w$ is decomposable as a right $R$-module (\Cref{cex1}).

For $W$ the universal Coxeter group on three generators, we find an element $w$ of length $6$ and an element $b\in \bar{B_w}$ for which there is an $R$-module morphism $\bbR\to \bar{B_w}$ sending $1$ to $b$ which is not induced by any morphism of bimodules $R\to B_w$ (\Cref{cex2}). This second counterexample is smaller and can be verified directly using the Sagemath code we attach.
These two counterexamples answer a question posed by Soergel in \cite[Remark 6.8]{SoeKazhdana} in the negative (see \Cref{lastremark}).

Most of the results in this paper originally appeared in the author's PhD thesis \cite{PatHodge}. Here, we streamline the arguments therein and adapt the language to more recent developments in the theory of Soergel bimodules. In particular, we use Abe's formulation of the category of Soergel bimodules \cite{AbeBimodule}, which allows us to work without the restrictive assumption that the realization $V$ is reflection faithful.

\section{Background and notation}

Let $(W,S)$ be a Coxeter system and $\Bbbk$ be a field. For $s,t\in S$ let $m_{st}$ denote the order of $st$.

\begin{definition}\label{realdef}
 A \emph{balanced realization}  of $W$ over $\Bbbk$ is a finite-dimensional representation $V$ of $W$ with the additional data of a subset of simple roots $\{ \alpha_s\}_{s\in S}\subset V$ and simple coroots $\{\alpha_s^\vee\}_{s\in S}\subset V^*$ such that for any $s,t\in S$ and $v \in V$ the following conditions hold. 
\begin{itemize} 
	\item $\alpha_s\neq 0$ and $\alpha_s^\vee\neq 0$.
	\item 	$\alpha_s^\vee(\alpha_s)=2$.
	\item $s(v)=v- \alpha_s^\vee(v) \alpha_s$.
	\item (\emph{balancedness}) 
	Let $x=-\alpha_s^\vee(\alpha_t)$ and $y=-\alpha_t^\vee(\alpha_s)$. Then we have $[m_{st}-1]_x=[m_{st}-1]_y=1$, where $[n]_x$ denotes the two-colored quantum number (cf. \cite[\S A.1]{EliTwo}).
\end{itemize}
\end{definition}

Let $T:=\bigcup_{w\in W} wSw^{-1}$ denote the set of reflections in $W$.
The balancedness condition ensures that we can unambiguously  associate a positive root to any reflection $t\in T$.

\begin{lemma}\label{balanced}
Let $r\in T$ be a reflection. Assume that $r=xsx^{-1}=yty^{-1}$ with $x,y\in W$, $s,t\in S$ and $xs>x$ and $yt>y$. Then $x(\alpha_s)=y(\alpha_t)$.
\end{lemma}
\begin{proof}
	This is proven for dihedral groups in \cite[\S 3.4]{EliTwo}. We prove it now for a general Coxeter group. Let $w:=y^{-1}x$ so that $wsw^{-1}=t$. We need to show that $w(\alpha_s)=\alpha_t$.
	
	Let $\Gamma$ be the odd Coxeter graph, the graph whose vertices are $S$ and with an edge between $s$ and $t$ if $m_{st}$ is odd. 	Two simple reflections $s$ and $t$ are conjugated to each other in $W$ if and only if they belong to the same connected component in $\Gamma$ (\cite[Lemma 3.6]{BMMWRigidity}).

	Let $\gamma=(s_0,s_1,\ldots,s_d)$ be a path in $\Gamma$, i.e. a sequence of elements $s_0,s_1,\ldots, s_d\in S$ such that $m_{s_is_{i+1}}$ is odd for any $i$. To $\gamma$, we associate the element 
	\[ \pi(\gamma):=\sigma_{d-1}\ldots\sigma_1\sigma_0\]
	where $\sigma_i$ is the longest element in the dihedral subgroup $\langle s_i,s_{i+1}\rangle$.
	
	Let $\gamma_{st}$ be a path in $\Gamma$ from $s$ to $t$.
	It follows that $w\pi(\gamma_{st})^{-1}\in Z(t)$, the centralizer of $t$ in $W$. From the dihedral case we have $\pi(\gamma_{st})(\alpha_s)=\alpha_t$.

	The centralizer $Z(t)$ of a simple reflection in a Coxeter group can be described as $Z(t)=\langle t\rangle\times Z'$, where $Z'$ is generated by elements of the form 
	\[\pi(\gamma_1,\gamma_2,u'):=\pi(\gamma_2)^{-1} (u'u)^{\frac12 m_{uu'}-1}u' \pi(\gamma_1)\] where $u\in S$, $\gamma_1,\gamma_2$ are paths in $\Gamma$ from $t$ to $u$ and $u'$ is such that $m_{uu'}$ is even (see \cite[Theorem 5]{AllReflection}).

The dihedral case implies that $\pi(\gamma_1,\gamma_2,u')$ fixes $\alpha_t$, so the same holds for all the elements in $Z'$.

		Notice that $t\pi(\gamma_{st})=\pi(\gamma_{st})s$. 
	Summarizing, we have $w = z' t^{\epsilon}\pi(\gamma_{st})=z' \pi(\gamma_{st})s^{\epsilon}$, with $z'\in Z'$ and $\epsilon\in \{0,1\}$, and	
	\begin{equation}\label{walphas}
w(\alpha_s)=\begin{cases} \alpha_t&\text{if }\epsilon=0,\\
	-\alpha_t &\text{if }\epsilon=1.\end{cases}
	\end{equation}
	 
	 It remains to show that $\epsilon=0$.
	The equation \eqref{walphas} holds for any balanced realization of $W$. In particular, it holds for the geometric representation. Calling temporarily $\alpha_s^{G}$, $\alpha_t^{G}$ the roots in the geometric representation, since $xs>x$ and $yt>y$ we know that $x(\alpha_s^G)$ and $y(\alpha_t^G)$ are positive roots (\cite[Prop. 4.2.5]{BBCombinatorics}). Hence, $w(\alpha_s^G)=\alpha_t^G$ and $\epsilon=0$.
\end{proof}

For $t\in T$ we can find $x\in W$ and $s\in S$ with $t=xsx^{-1}$ and $xs>x$. We define $\alpha_t:=x(\alpha_s)$. This is well-defined by \Cref{balanced}.

\begin{remark}
	There has been a lot of work over the years to extend the definition of Soergel bimodules to realizations having weaker conditions. In the original definition \cite{SoeKazhdana}, the assumption was that the representation is reflection faithful. In \cite{AbeBimodule}, Abe showed that it is enough to assume that it is reflection faithful on diehdral subgroups. In \cite{AbeHomomorphism} he further weakened this assumption, and replaced it by an identity on two-colored quantum numbers that $\alpha_s^\vee(\alpha_t)$ and $\alpha_t^\vee(\alpha_s)$, with $s,t\in S$, are required to satisfy. This condition is always satisfied if the representation is dihedrally faithful, i.e. if the restriction to any subgroup generated by two simple reflections is faithful (cf. \cite[Prop. 3.6]{AbeHomomorphism}).
\end{remark}

	Here, we cannot work in this full generality since we also require that the category of sheaves in the moment graph is well-behaved. To ensure this, we impose an additional assumption, which is known as the GKM condition:
	
	\begin{equation}\tag{GKM}
		\text{For any $t,t'\in T$ with $t\neq t'$ the roots $\alpha_t$ and $\alpha_{t'}$ are linearly independent.}
	\end{equation}
	
	\begin{definition}
		We say that $V$ is a \emph{GKM-realization} if it satisfies the assumptions in \Cref{realdef} plus the GKM condition.
	\end{definition}
	
	The GKM condition is sufficient to ensure that Abe's assumptions are satisfied.

\begin{lemma}\label{dfaithful}
	Any GKM-realization is dihedrally faithful.
\end{lemma}
\begin{proof}
	
	Let $V$ be a GKM-realization of $W$.
	Let $s,t\in S$ and let $W'$ be the subgroup generated by $s$ and $t$. 
	Assume that there is $w\in W'$ acting trivially on $V$.
	
	 If $\ell(w)$ is odd, then
	 then $w$ must be conjugate within $W'$ to either s or t. Since $w$ acts trivially, its conjugate ($s$ or $t$) must also act trivially on $V$.
	  But this contradicts the assumption that $\alpha_s^\vee,\alpha_t^\vee\neq 0$.

	Assume that $\ell(w)$ is even. Then $w$ is the product of a simple reflection and of a reflection in $W'$, say $w=su$ with $u\in T$. So $u$ acts on $V$ as $s$. Since $\alpha_s^\vee\neq 0$, we can find $v \in V$ such that $s(v)-v \in \Bbbk^* \alpha_s$. From this, it follows that $\alpha_u$ and $\alpha_s$ are linearly dependent. The GKM condition now forces that $u=s$ and $w=e$.
\end{proof}

Let \( R:= \Sym(V) \) be the ring of regular functions on \( V^* \).
We regard \( R \) as a graded ring by setting \( \deg(V) = 2 \). 
We denote by \( R_+ \) the ideal of \( R \) generated by homogeneous polynomials of positive degree. 
We view \( \Bbbk \) as an \( R \)-module via \( \Bbbk \cong R / R_+ \). 

The action of \( W \) on \( V \) extends to an action on \( R \). 
For a reflection \( t \in T \) , we denote by \( \partial_t : R \to R \) the corresponding \emph{Demazure operator}, defined by
\[
\partial_t(f) = \frac{f - t(f)}{\alpha_t}.
\]

\section{The structure algebra of a Coxeter group}

\subsection{The nil Hecke ring and its dual}\label{nilhecke}

The nil Hecke ring was defined 
 by Kostant and Kumar in \cite{KKNil} and Arabia in \cite{AraCohomologie}.
It serves as an algebraic construction of the equivariant cohomology of the flag variety of a reductive group.
It is important to remark that its construction can be generalized to arbitrary Coxeter groups (as pointed out in \cite[Remark 4.35(b)]{KKNil})
In fact, Kostant and Kumar's original motivation was to provide an algebraic description of the (equivariant) cohomology of flag varieties of Kac--Moody groups.
We follow here the treatment in \cite{RZNil}.

Let $V$ be a GKM-realization of $W$ and let $R=\Sym(V)$. 
Let $Q=R[\frac{1}{\alpha_t}\mid t \in T]$\label{s:Q} be the localization of $R$ at the set of roots. Let $Q_W$ denote the smash product of $Q$ with $W$. This means that $Q_W$ is a free left $Q$-module with basis $\{\delta_w\}_{w\in W}$ and multiplication defined by
$$(f\delta_x)(g\delta_y)=fx(g)\delta_{xy}.$$
In particular, $f\delta_x=\delta_x x^{-1}(f)$. 
There is an action of $W$ on $Q_W$ defined by conjugating by $\delta_w$:
\[ w(f\delta_x):= \delta_w f\delta_x \delta_{w^{-1}}= w(f)\delta_{wxw^{-1}}.\]

For $s\in S$ we define the element $$D_s=\frac{1}{\alpha_s}(\delta_{\eid}-\delta_s)=(\delta_\eid+\delta_s)\frac{1}{\alpha_s}\in Q_W.$$
We have $D_s^2=0$ and the $D_s$ satisfy the braid relations \cite[Proposition 4.2]{KKNil}\footnote{The argument in \cite{KKNil} assumes the representation $W$ to be faithful, but to show the braid relations it is enough to work within the dihedral subgroup $\langle s,t\rangle$, for which faithfulness holds by \Cref{dfaithful}}, i.e.
$$\underbrace{D_sD_tD_s\ldots}_{m_{st} \text{times}}=\underbrace{D_tD_sD_t\ldots}_{m_{st} \text{times}}$$

Hence, for $x\in W$ we can define $D_x=D_{s_1}D_{s_2}\ldots D_{s_l}$ where $\undx=s_1s_2\ldots s_{l}$ is any reduced expression for $x$. We have a natural left action of $Q_W$ on $Q$ via $f\delta_x\cdot g=fx(g)$.\label{s:D_x}

We call $R$-ring a ring which is also a module over $R$.
\begin{definition}
	The \emph{nil-Hecke ring} $\frN$ is the $R$-subring of $Q_W$ generated by $\{D_s\}_{s\in S}$.\footnote{By $R$-subring generated by $\{D_s\}$ we mean the smallest subring of $Q_W$ containing $D_s$ and which is closed under left multiplication by $R$. Notice that this is not an $R$-algebra because the action of $R$ is not central.}
\end{definition}

Observe that  $\delta_s=-\alpha_s D_s+\delta_{id} \in \frN$, so $\delta_w\in \frN$ for any $w\in W$. Moreover, $\frN$ is stable under the action of $W$.


\begin{theorem}[{\cite[Proposition 3.11]{RZNil}}]
	The ring $\frN$ is a free right $R$-module with basis $\{D_w\}_{w\in W}$.
\end{theorem}

We can also describe the ring $\frN$ by generators and relations (cf. \cite[Definition 3.3]{RZNil}). The nil-Hecke ring $\frN$ is the $R$-ring with generators $\{D_s\}_{s\in S}$ and relations
\[ D_s^2=0\qquad \overbrace{D_sD_t\cdots}^{m_{st}} =\overbrace{D_t D_s \cdots}^{m_{st}}\]
\[ D_s \lambda= s(\lambda) D_s+\alpha_s^\vee(\lambda).\]

In general, the product of an element $D_w$ with $\lambda\in V$ can be computed as follows.
\begin{equation}\label{pieri} D_w \cdot p = w(\lambda)D_w + \sum_{v\xto{t} w}\alpha_t^\vee(\lambda)D_v\end{equation}
where $t\in T$ and $v\xto{t} w$ means $w=vt$ and $\ell(w)=\ell(v)+1$.

Let $Q_W^*=\Hom_{Q-}(Q_W,Q)$ be the set of left $Q$-module morphisms. We can think of $Q_W^*$ as the set of functions $W\raw Q$, where to an element $\psi\in Q_W^*$ corresponds the function $W\raw Q$ which sends $x\in W$ to $\psi(\delta_x)$. We regard $Q_W^*$ as a $Q$-ring, via point-wise addition, scalar multiplication and multiplication, that is if $f,g \in Q_W^*$ then $f\cdot g(\delta_w)=f(\delta_w)g(\delta_w)$. 

Notice that $\{D_w\}_{w\in W}$ is also a basis of $Q_W$ as a left $Q$-module. 	Let $\{\xi^w\}_{w\in W}\subset Q_W^*$ be the dual basis of $\{(-1)^{\ell(w)}D_w\}$. That is, $\xi^w$ is defined by 
\[ \xi^w(D_x)=(-1)^{\ell(x)}\delta_{x,w}.\]



	\begin{definition}
	The \emph{dual nil-Hecke ring} $\frN^*$ is the 
	subalgebra of $Q_W^*$ generated by $\xi^w$.
	\end{definition}
	
	We have that $\xi^u\cdot \xi^v=\sum p_{u,v}^w \xi^w$ for some $p_{u,v}^w\in R$ by \cite[Theorem 4.2]{RZNil}. It follows that $\{\xi^w\}_{w\in W}$ is a basis of $\frN^*$. 
	
	Let $d_{x,y}:=\xi^x(\delta_y)\in R$. For $w\in W$ let 
	\begin{equation}\label{pw}p_w=\prod_{\stackrel{t\in T}{tw<w}}\alpha_t.\end{equation}

		\begin{lemma}\label{dxy}
				For $x\in W$ we have
				\begin{enumerate}\item $d_{x,y}=0$ unless $x\leq y$. 
					\item $d_{x,x}=p_x$ and $\deg(d_{x,y})=2\ell(x)$.
					\item
		For any reflection $t$ and any $x,y\in W$ we have $d_{x,y}\equiv d_{x,ty} \pmod{\alpha_t}$.
	\end{enumerate}
	\end{lemma}
	\begin{proof}
		If $w=s_{i_1}\ldots s_{i_r}$ then $p_w=\alpha_{i_1}\cdot s_{i_1}(\alpha_{i_2}) \cdots s_{i_1} s_{i_2} \cdots s_{i_{r-1}} (\alpha_{i_r})$ and 
		\begin{equation}\label{Dwdeltaw}
		D_w = \frac{1}{\alpha_{i_1}} (\delta_e - \delta_{s_{i_1}}) \cdots \frac{1}{\alpha_{i_r}} (\delta_e - \delta_{s_{i_r}})= \frac{(-1)^{\ell(w)}}{p_w} \delta_w +
		\sum_{u<w} c_{w,u} \delta_u,
		\end{equation}
		with $c_{w,u}\in Q$ homogeneous of degree $-2\ell(w)$.
		From this we deduce \begin{equation}\label{D=d}\bigoplus_{w\leq y} Q \cdot D_w=\bigoplus_{w\leq y} Q \cdot \delta_w\end{equation} for any $y \in W$. If $x\not \leq y$, then $\xi^x$ vanishes on \eqref{D=d} and $(1)$ follows.
		Moreover, we have by \eqref{Dwdeltaw}  that $\xi^x(D_x)=\frac{(-1)^{\ell(x)}}{p_x}d_{x,x}$, from which $d_{x,x}=p_x$. Similarly, by induction on $\ell(w)$, we see that $\xi^x(\delta_w)$ is homogeneous of degree $2\ell(x)$ and $(2)$ follows.		
		
		We now prove $(3)$. Given a reflection $t\in T$,
		we can choose $w\in W$ and $s\in S$ such that $wsw^{-1}=t$ and $w(\alpha_s)=\alpha_t$. We have
		\[ w(D_s)=w(\frac{1}{\alpha_s}(1-\delta_s))=\frac{1}{\alpha_t}(1-\delta_t)=:D_t\]	
		
		and \[d_{x,y}-d_{x,ty}=\xi^x(\delta_{y}-\delta_{ty})=\xi^x(\alpha_t D_t \delta_y)=\alpha_t \xi^x(D_t\delta_y).\]
		We conclude since $D_t\delta_y\in \frN$.
	\end{proof}
	
	There is also an action of $W$ on $\frN^*$ defined by $
	x\cdot \psi(Y)=\psi(Y\delta_x)$ for any $Y\in Q_W$. 

\begin{lemma}\label{lemmasxi}
	For $s \in S$ we have
	\[ \displaystyle s\cdot \xi^w=\begin{cases}
		\xi^w &\text{if }ws>w\\
\xi^w - w(\alpha_s)\xi^{ws}-\sum_{ws\xto{t} v} \alpha_t^\vee(\alpha_s)\xi^v&\text{if }ws<w	\end{cases}\]
\end{lemma}
\begin{proof}
	We have $\delta_s=-\alpha_s D_s+\delta_{id} $, and
	\begin{align*} D_v\delta_s &=-D_v\alpha_s D_s+D_v \\
		&=-v(\alpha_s)D_v D_s-\sum_{u\xto{t}v}\alpha_t^{\vee}(\alpha_s)D_uD_s+D_v.
	\end{align*}	
We consider first the case $ws>w$. Then $s\cdot \xi^w(D_v)=\xi^w(D_v\delta_s)=(-1)^{\ell(w)}\delta_{v,w}$ because $D_w$ cannot be expressed in the form $D_uD_s$, so $s\cdot\xi^w=\xi^w$.	
	
Assume now $ws<w$. We have
\[(-1)^{\ell(w)}\xi^w(D_v\delta_s)=\begin{cases}
	1 -\alpha_s^\vee(\alpha_s)&\text{if }v=w\\
	-v(\alpha_s)& \text{if }v=ws\\
	-\alpha_t^\vee(\alpha_s)& \text{if }vts=w \text{ with }\ell(v)=\ell(w)\\
		0 &\text{otherwise}\\
	\end{cases}\]
from which we get
\[ s\cdot \xi^w=\xi^w-w(\alpha_s)\xi^{ws}-\sum_{ws\xto{t} v} \alpha_t^\vee(\alpha_s)\xi^v.\qedhere\]
\end{proof}

\begin{cor}\label{corsxi}
	The subspace of $s$-invariants $(\frN^*)^s$ is a free $R$-module with basis $\{\xi^w\}_{ws>w}$.
\end{cor}
\begin{proof}
	We know from \Cref{lemmasxi} that $\xi^w\in (\frN^*)^s$ if $ws>w$. Assume now that there exists $0\neq \sum_{ws<w} p_w \xi^w\in (\frN^*)^s$. Let $y$ be of minimal length in the sum with $p_y\neq 0$. Then in $s\cdot \left(\sum_{ws<w} p_w \xi^w\right)$ the coefficient of $\xi^{ys}$ is $-p_y \cdot y(\alpha_s)$, from which we obtain $p_y=0$, which is a contradiction.
\end{proof}

	\subsection{Moment graphs of Coxeter groups}\label{moment}

There exists another description of the equivariant cohomology of the flag variety, obtained by Goresky, Kottwitz and MacPherson \cite{GKMEquivariant} using the localization theorem for torus actions. 
As pointed out by Fiebig \cite{FieCombinatorics}, one can generalize this construction to an arbitrary Coxeter group. We show that for an arbitrary Coxeter group and for a realization $V$ satisfying our assumptions this construction still returns the dual nil-Hecke ring.
%
%
%
%

%
%

The \emph{unbounded structure algebra} $\hat{Z}$ is defined by

$$\hat{Z}=\left\{(r_v)\in \prod_{v\in W}R\mid r_v\equiv r_{tv}\pmod{\alpha_t}\;\forall v\in W, t \in T\right\}.$$\label{s:hatZ}

For $i\in \bbN$ let $\hat{Z}_i$ be the graded component of $\hat{Z}$, that is $\hat{Z}_i:=\{(z_v)\in \hat{Z} \mid \deg{z_v}=i\}$. 
We define $Z:=\bigoplus_{i\in \bbZ} \hat{Z}_i$. Then $Z$ is a subring of $\hat{Z}$. We call $Z$ the \emph{(bounded) structure algebra}. We can also describe $Z$ as the subring of sections in $\hat{Z}$ with bounded degree, that is $Z=\{(z_v)\in \hat{Z} \mid \exists i:\deg{z_v}\leq i$ for all $v\in W\}$.\label{s:Z}
Notice that for an infinite Coxeter group we have $Z\neq \hat{Z}$. 

By \Cref{dxy}, we can define an element of the structure algebra $\calP_x\in Z$ by setting $(\calP_x)_y:=d_{x,y}$.
The element $\calP_x$ is homogeneous of degree $2\ell(x)$.

\begin{lemma}
	The set $\{\calP_x\}_{x\in W}$ is a basis of $Z$ as an $R$-module.
\end{lemma}
\begin{proof}
	By a triangularity argument, the set $\{\calP_x\}_{x\in W}\cug Z$ is linearly independent over $R$ since $(\calP_x)_y=0$ for $y\not \geq x$ and $(\calP_x)_x=p_x\neq 0$ by \Cref{dxy}.
	
	Let $Z'=\spa_R\langle \calP_x\mid x\in W\rangle$. Let $f\in Z$ be homogeneous of degree $2d$. We need to show that $f\in Z'$. We fix an enumeration $w_1,w_2,w_3\ldots$ of the elements of $W$ which refines the Bruhat order.
	Let $h$ be minimal such that $f_{w_h}\neq 0$. 
	Then $p_{w_h}| f_{w_h}$ because for any $t\in T$ with $tw_h<w_h$ we have $\alpha_t \mid f_{w_h}$. So we can replace $f$ with $$f':=f-\frac{f_{w_h}}{p_{w_h}}\calP_v\in Z$$
	Notice that $f'$ is homogeneous of degree $2d$ and $f'_{w_i}=0$ for all $i\leq h$. 
	If we repeat this enough times we end up with $g\in Z'$ of degree $2d$ such that $g_x=0$ for all $x\in A_d$, where $A_d:=\{x\in W\mid \ell(x)\leq d\}$.
	
	Assume now $g\neq 0$, so there exists a minimal element $w\in W$ such that $\ell(w)>d$ and $g_w\neq 0$. But this would imply $p_w|g_w$, which is impossible since $\deg p_w>2d$.
\end{proof}

There is a natural $W$-action on $Z$. For $z\in \prod_{v\in W}R$ and $x\in W$ we define $(x\cdot z)_v=z_{vx}$. This action preserves $Z$: in fact $\alpha_t$ divides $(x\cdot z)_{tv}-(x\cdot z)_v=z_{tvx}-z_{vx}$ for all $t\in T$, $x,v,w\in W$.

\begin{theorem}\label{lambda=z}
	There exists a $W$-equivariant isomorphism of graded $R$-rings $\Phi:\frN^*\xra{\sim}Z$ which sends $\xi^x\in \Lambda$ to $\calP_x\in Z$.
\end{theorem}
\begin{proof}
	For $\psi\in \frN^*$ we can define $\Phi(\psi)=(\psi(\delta_{x}))_{x\in W}\in \prod_{x\in W}R$. The map $\Phi$ is a homomorphism of $R$-rings from $\frN^*$ to $\prod_{x\in W}R$.
	
	We have $\Phi(\xi^x)_y=d_{x,y}=(\calP_x)_y$. In particular, $\Phi(\Lambda)\cug Z$ and, as a morphism  of left $R$-modules, it sends a basis into a basis, hence it is an isomorphism.
	
	For any $x,y \in W$ we have 
	\[(\Phi(x\cdot \psi))_y=(x\cdot \psi)(\delta_{y})=\psi(\delta_{yx})=\Phi(\psi)_{yx}=(x\cdot \Phi(\psi))_y.\]
	 It follows that $\Phi(x\cdot \psi)=x\cdot \Phi(\psi)$.
\end{proof}

There is also a right action of $R$ on $Z$ defined by $(z_x)_{x\in W} \cdot f=(x(f)z_x)_{x\in W}$. This corresponds, via the isomorphism $\Phi$, to the action on $\frN^*$ defined by $(f\bullet \psi)(v):=\psi(vf)$.
For $\lambda \in V$, we have 
\[ \lambda\bullet \xi^w=w(\lambda)\xi^w -\sum_{w\xto{t}v} \alpha_t^\vee(\lambda) \xi_v.\]
as it follows immediately from \eqref{pieri}. So we also obtain
\begin{equation}\label{PieriZ}
	\calP_{w}\cdot \lambda=w(\lambda)\calP_{w}+\sum_{w\xto{t}v}\partial_t(\lambda)\calP_{v}.
\end{equation}


For a subset $\Omega\cug W$ we define $\hat{Z}^\Omega$ to be the image of the composition 
$$\hat{Z}\hookrightarrow \prod_{v\in W}R\twoheadrightarrow \prod_{v\in \Omega}R.$$
We define $Z^\Omega$ similarly. Clearly, for any finite subset $\Omega$ we have $Z^\Omega=\hat{Z}^\Omega$. 
A subset $\Omega$ is said to be \emph{upwardly closed} if whenever $v\in \Omega$ and $w\geq v$, then $w\in \Omega$.

\begin{definition}
	Let $Z$-mod${}^f$ be the full subcategory $Z$-mod whose objects are graded $Z$-modules $M$ which are finitely generated and torsion free over $R$ and such the $Z$-module structure factors through $Z^\Omega$ for some finite $\Omega\cug W$.
\end{definition}

We define similarly $\hat{Z}$-mod${}^f$. The restriction functor $\hat{Z}$-mod$\raw Z$-mod induces an equivalence of categories $\hat{Z}$-mod${}^f\xra{\sim} Z$-mod${}^f$.

For any $s\in S$, we set
\[Z^s :=
\left\{
(z_w) \in Z
\;\middle|\;
z_{ws} = z_w \text{ for any } w \in W 
\right\}.\]

\begin{lemma}\label{Zsbasis}
The subring $Z^s$ is the subring of $Z$
consisting of $s$-invariants. As a left $R$-module, it has a basis given by $\{\calP_w\}_{ws>w}$.

Let $\varpi_s\in V$ be such that $\alpha_s^\vee(\varpi_s)=1$. Then $Z$ is free as a $Z^s$ module with basis $1$ and $(w(\varpi_s))_{w\in W}$.
\end{lemma}
\begin{proof}
	The first statement follows from \Cref{lambda=z} and \Cref{corsxi}. The second statement follows from \cite[Lemma 5.2]{AbeBimodule}.
\end{proof}

\subsection{Relationship with Soergel bimodules}\label{sbim}

For a graded module $M$ and $i\in \bbZ$ let $M(i)$ denote the shifted module, i.e. $(M(i))^k=M^{i+k}$.

For $s \in S$ we denote by $B_s$ the graded $R$-bimodule $R\otimes_{R^s}R(1)$. 
Let $\undw=s_1s_2\ldots s_k$ be an expression, not necessarily reduced. The Bott--Samelson bimodule $\BS(\undw)$ is the graded $R$-bimodule defined as
$$\BS(\undw)=B_{s_1}\otimes_R B_{s_2}\otimes_R\ldots \otimes_R B_{s_k}=R\otimes_{R^{s_1}} R\otimes_{R^{s_2}}R\otimes \ldots\otimes_{R^{s_k}} R(k).$$ 

We denote the element $1\otimes \ldots \otimes 1 \in \BS(\undw)$ by $1^\otimes_\undw$.

\begin{definition}
	Let $\calC$ be the full subcategory of finitely generated graded $R$-bimodules $M$ which are flat as left $R$-modules, equipped with a decomposition $Q\otimes_R M=\bigoplus_{w\in W} M_Q^w$ of graded $(Q,R)$-bimodules such that
	\begin{itemize}
		\item $M_Q^w=0$ for all but finitely many $w\in W$,
		\item for $m\in M_Q^w$ we have $mf=w(f)m$.
	\end{itemize} 
\end{definition}

The bimodules $B_s$ have a unique decomposition $(B_s)_Q=(B_s)_Q^e\oplus (B_s)_Q^s$ (cf. \cite[\S 2.4]{AbeBimodule}). By taking tensor products, this induces a canonical decomposition for $Q\otimes_R \BS(\undw)$. We can then regard $\BS(\undw)$ as an object in $\calC$ together with this decomposition.

\begin{remark}
	 In the original definition in \cite[\S 2.2]{AbeBimodule}, $Q$ is the ring of fractions of $R$. However, as noted in \cite[Remark 3.8]{AbeHomomorphism}, to ensure that $M^w_Q$ is a graded module, $Q$ must be  taken to be the localization of $R$ at the roots.
\end{remark}

\begin{definition}
	The category of Soergel bimodules $\sbim$ is the smallest full subcategory of $\calC$ that contains all the Bott--Samelson bimodules $\BS(\undw)$ for any expression $\undw$ and that is closed under grading shifts, finite direct sums and direct summands.
\end{definition}

\begin{remark}
	If the representation $V$ of $W$ is faithful, then the decomposition $Q\otimes B=\bigoplus B_Q^x$ can be retrieved directly from $R$-bimodule structure. In this case, $\sbim$ is a subcategory of the category of $R$-bimodules. (cf. \cite[Remark 2.2]{AbeBimodule}).
	In general, though, giving such a decomposition requires additional data.
\end{remark}

Morphisms in $\sbim$ are degree-preserving morphisms $B\to B'$ of $R$-bimodules, i.e., homogeneous of degree $0$, which send $B_Q^w$ to $(B')_Q^w$. For $B,B'\in \sbim$ we write
$$\Hom^\bullet(B,B')=\bigoplus_{i\in \bbZ}\Hom_{\sbim}(B,B'(i)).$$

Recall that we have a left and right $R$-action on $Z$. By restriction, we can regard any $Z$-module as an $R$-bimodule. We denote by $F:Z$-Mod$\to R$-bimod the restriction functor.

Define a $Z$-module structure on $R$ by $(z_w)_{w\in W} f = z_ef$ for $(z_w)_{w\in W} \in Z$ and $f \in R$ and
denote this $Z$-module by $R_e$. Then $F(R_e) = R$. Let $Z$-mod${}^S$ be the full-subcategory
of $Z$-mod${}^f$ consisting of the direct summands of direct sums of $Z \otimes_{ Z^{s_1}} \cdots\otimes_{Z^{s_l}} R_e(n)$ with 
$s_1, \ldots , s_l\in S$ and $n \in \bbZ$.

\begin{theorem}[{\cite[Theorem 5.4]{AbeBimodule}}]\label{equivZSBim}
	Assume that $V$ is a GKM-realization. The functor $F$ induces an
	equivalence $Z$-mod${}^S \rightarrow \sbim$.
\end{theorem}

Let $[\sbim]$ denote the split Grothendieck group of the category of Soergel bimodules. We consider $[\sbim]$ as a $\bbZ[v,v^{-1}]$ algebra via $v\cdot [B]=[B(1)]$.
The tensor product $\otimes_R$ equips the category $\sbim$ with a monoidal structure, which induces a $\bbZ[v,v^{-1}]$-algebra structure on $[\sbim]$. 

\begin{theorem}[{Soergel's Categorification Theorem, \cite[Theorem 4.1]{AbeBimodule}}]\label{sct}
	
	Let $w\in W$ and let $\undw$ be a reduced expression for $w$. Then there exists a unique direct summand $B_w$ of $\BS(\undw)$ which is not a summand of $\BS(\undv)$ for any expression $\undv$ with $\ell(\undv)< \ell(w)$.
	Moreover, the summand $B_w$ can be characterized in any decomposition of $\BS(\undw)$ as the indecomposable summand containing $1^\otimes_\undw$.
\end{theorem}

Let $\mathcal{H}(W,S)$ be the Hecke algebra of $W$. This is the algebra over $\bbZ[v,v^{-1}]$ with basis $\{\bfH_w\}_{w\in W}$  and relations:
\begin{itemize}
	\item $(\bfH_s-v^{-1})(\bfH_s+v)=0$ for any $s \in S$.
	\item $\bfH_{w}\bfH_s=\bfH_{ws}$ for any $w\in W$ with $ws>w$.
\end{itemize}

\begin{theorem}[Soergel's hom formula
	 {\cite[Theorems 4.3 and 4.6]{AbeBimodule}}]\label{Hiso}
	There exists an isomorphism of $\bbZ[v,v^{-1}]$-algebras $\mathcal{E}: [\sbim]\raw\mathcal{H}(W,S)\xra{\sim}$ such that $[B_s]\mapsto \bfH_s+v$.

	Let $B,B'\in \sbim$. Then $\Hom^\bullet(B,B')$ is a graded free left $R$-module, and 
	$$\grrk \Hom^\bullet(B,B')=(\bar{[B]},[B'])$$
	where $(-,-)$ is the standard pairing of the Hecke algebra (cf.  \cite[Definition 3.13]{EMTWIntroduction}) and $\bar{(-)}$ denotes the Kazhdan--Lusztig involution (cf.  \cite[Definition 3.13]{EMTWIntroduction}).
\end{theorem}

If $B$ is a  self-dual bimodule (that is, if  $\bar{[B]}=[B]$), the pairing can be expressed in simpler terms (cf. \cite[Lemma 3.19]{EMTWIntroduction}). If $[B] = \sum c_x \bfH_x$ and $[B']= \sum d_x \bfH_x$, then 
\[\grrk\Hom^{\bullet}(B,B')=(\bar{[B]},[B'])=
\sum_{x\in W} c_xd_x.\]

\section{Cohomology and Homology within Soergel bimodules}

\subsection{Light leaves basis}

Let $\undw=s_1\ldots s_\ell$ be an expression of length $\ell$ and let $\unde\in \{0,1\}^\ell$. We write $\undw^\unde = s_1^{e_1}\cdots s_\ell^{e_\ell}$ and we say that $\unde\subset \undw$ is a subexpression of $\undw$ with target $x=\undw^\unde$.

For $i\leq \ell$, let $x_i=s_1^{e_1}\ldots s_\ell^{e_i}$.
We decorate the sequence $\unde$ with a sequence of $U$ and $D$ as follows. At the index \( i \) we decorate with $U$ if \( x_{i-1}s_i > x_{i-1} \), and with \( D \) otherwise. 
The \emph{defect} \( \defect(\unde) \) of \( \unde \) is defined by 
\[
\defect(\unde) := \#\{i \mid \text{label at }i \text{ is } U, \, e_i = 0\} - \#\{i \mid \text{label at } i \text{ is } D, \, e_i = 0\}.
\]

We say that $x\leq \undw$ if there exists $\unde\subset \undw$ with $\undw^\unde=x$. For any $x\leq \undw$, by \cite[Prop. 12.20]{EMTWIntroduction} there exists a unique subexpression $\can_x\subset \undw$ such that $\can_x$ is decorated only with $U$'s and $\undw^{\can_x}=x$. We call $\can_x$ the \emph{canonical subexpression} for $x$.

For each subexpression $\unde\subset \undw$, Libedinsky in \cite{LibSur} defined a morphism $\LL_{\undw,\unde}=\Hom^{\defect(\unde)}(\BS(\underline{w}),\BS(\undx))$, where $\undx$ is a reduced expression for $x$ (see also \cite[\S 10.4]{EMTWIntroduction} and \cite[Definition 3.8]{AbeBimodule}). The morphism $\LL_{\undw,\unde}$ is not uniquely identified but depends on several choices of reduced expressions  made in its construction.
Its value on $1_\undw^\otimes$ does not depend on these choices.

\begin{lemma}\label{D=0}
	Let $\undw$ be an expression and $\unde\subset \undw$ be a subexpression. Then
	$$\LL_{\undw,\unde}(1^\otimes_{\undw})=
	\begin{cases}1^\otimes_{\underline{x}} & \text{ if } \unde=can_{\undw^\unde},\\0 & \text{ if } \unde\text{ has at least one }D.\end{cases}$$
\end{lemma}
\begin{proof}
	See \cite[Proposition 12.19]{EMTWIntroduction} or \cite[Proposition 3.10]{AbeBimodule}.
\end{proof}

\subsection{Invariant forms and duality of Soergel bimodules}\label{forms}

We define the \emph{dual} $\bbD B$ of $B\in \sbim$ to be $\bbD B=\Hom_{R-}^\bullet(B,R)$, where $\Hom_{R-}^\bullet(-,-)$ denotes the space of morphisms of left $R$-modules of all degrees, together with the decomposition
$\bbD(B)_Q=\bigoplus \bbD(B)_Q^w$, where $\bbD(B)_Q^w=\Hom_{Q-}(B_Q^w,Q)$ 
We can give $\bbD B$ a structure of a graded $R$-bimodule via $r_1fr_2(b)=f(r_1 b r_2)$, for any $f\in \bbD B$, $b\in B$ and $r_1,r_2\in R$.
(cf. \cite[Lemma 2.20]{AbeBimodule}).

\begin{definition}\label{lif}
A \emph{left invariant form} on $B\in \sbim$ is a homogeneous bilinear form 
$$\langle -,-\rangle: B\times B\raw R$$
satisfying the following conditions for all $b,b'\in B$ and $f\in R$.
\begin{itemize}
	\item $\langle b,b'f\rangle=\langle bf,b'\rangle$.
	\item$\langle fb,b'\rangle=\langle b,fb'\rangle=f\langle b,b'\rangle$.
	\item For any $x, y\in W$ with $x\neq y$ we have $\langle B_Q^x,B_Q^y\rangle =0$.
\end{itemize}
\end{definition}

A left invariant form on $B$ is the same data as a morphism $B\to \bbD(B)$ (cf. \cite[Proposition 18.9]{EMTWIntroduction})
We say that a pairing is \emph{non-degenerate} if the induced map $B\raw \bbD B$ is an isomorphism.

Let $\varpi_s\in V$ be such that $\alpha_s^\vee(\varpi_s)=1$, and let
\[c_s:=s(\varpi_s)\otimes 1-1\otimes \varpi_s\in B_s\qquad\text{and}\qquad \tilde{c}_s:=\varpi_s\otimes 1-1\otimes \varpi_s\in B_s\] 
The element $c_s$ is, up to scalar factor, the unique element of $B_s$ of degree $1$ such that $fc_s=c_sf$ for all $f\in R$. Similarly, the element $\tilde{c}_s$ is the unique element such that $s(f)\tilde{c}_s=\tilde{c}_sf$ for $f\in R$.
The map $R\raw B_s$ which sends $1$ to $c_s$ is a homomorphism of $R$-bimodules.

	Let $c_{\eid}=1\otimes 1\in B_s$. The set $\{c_\eid,c_s\}$ is a basis of $B_s$ as a left $R$-module. By abuse of notation we write $c_s^1=c_s$ and $c_s^0=c_\eid$. 
	
	Let $\undw=s_1s_2\ldots s_k$. For $\unde \in \{0,1\}^k$ we define \begin{equation}\label{c_e}c_{\unde}:=c_{s_1}^{e_1}\otimes c_{s_2}^{e_2}\otimes\cdots \otimes c_{s_k}^{e_k}.\end{equation} The set $\{c_{\unde}\mid e$ a $01$-sequence for $\undw\}$ is a basis of $\BS(\undw)$ as a left $R$-module. We denote $c_{00\ldots 0}$ by $c_{top}$. Notice that $c_{11\ldots 1}=1\otimes 1\otimes \ldots \otimes 1=1^\otimes_\undw$. We call this set the \emph{string basis} of the Bott--Samelson bimodule.
	
	Notice that a Bott--Samelson bimodule is a shifted graded algebra\footnote{By shifted graded algebra we mean an algebra $A$ such that $A(d)$ is a graded algebra in the usual sense for some $d\in \bbZ$.} with respect of component-wise multiplication, and $1^\otimes_\undw$ is its (shifted) unity of degree $-\ell(\undw)$.
	Let $\Trace:\BS(\undw)\raw R$ be the left $R$-linear map which returns the coefficient of $c_{top}$ in the string basis.
	Let 
	\begin{equation}\label{intform}\langle f,g\rangle_{\BS(\undw)}=\Trace(f\cdot g),\end{equation} where $f\cdot g$ stands for the multiplication in $\BS(\undw)$. 
	
	\begin{lemma}\label{nondeg}
	The pairing $\langle-,-\rangle_{\BS(\undw)}$ is left invariant and it is non-degenerate. 
	\end{lemma}
	\begin{proof}
		The pairing $\langle-,-\rangle_{\BS(\undw)}$ satisfies the first two conditions in \Cref{lif} and it is non-degenerate by {\cite[Proposition 12.18]{EMTWIntroduction}}. It remains to show the third condition in \Cref{lif}, i.e., that it is compatible with the decomposition $\BS(\undw)_Q=\bigoplus_{w\in W}\BS(\undw)_Q^w$. We prove this by induction on $\ell(\undw)$.
		
		If $\ell(\undw)=0$ there is nothing to show. Let $\undw=\undv s$ and let $B=\BS(\undv)$. 	By \cite[Lemma 2.11]{AbeBimodule} we have
		\[(B\otimes B_s)_Q^x=\{(m\otimes c_s+m'\otimes \tilde{c_s}\mid m\in B_Q^x,\, m'\in B_Q^{xs}\}\]
	Since $\varpi_s s(\varpi_s), s(\varpi_s)+\varpi_s\in R^s$, we have	\begin{align*}c_s\cdot \tilde{c}_s&=\varpi_s s(\varpi_s)\otimes 1-(s(\varpi_s)+\varpi_s)\otimes \varpi_s-\varpi_s\otimes \varpi_s+1\otimes \varpi_s^2\\
			&=1\otimes \varpi_s s(\varpi_s)-1\otimes (s(\varpi_s)+\varpi_s)\varpi_s+1\otimes \varpi_s^2=0
		\end{align*} 
		Let $\tau=\langle c_s,c_s\rangle_{B_s}$ and $\tilde{\tau}=\langle \tilde{c}_s,\tilde{c}_s \rangle_{B_s}$.
		Then, for $b\in \BS(\undw)^x$ and $b'\in \BS(\undw)^y$ with $b=m\otimes c_s+m\otimes \tilde{c}_s$ and $b'=n\otimes c_s+n'\otimes \tilde{c_s}$, with $m\in B_Q^x$, $m'\in B_Q^{xs}$, $n\in B_Q^y$ and $n'\in B_Q^{ys}$, we have	
		\begin{align*}\langle b,b'\rangle_{\BS(\undw)}&=\langle m\otimes c_s+m'\otimes \tilde{c}_s, n\otimes c_s+n'\otimes\tilde{c}_s\rangle_{\BS(\undw)}\\&=\langle m,n\tau\rangle_{\BS(\undv)}+\langle m',n'\tilde{\tau}\rangle_{\BS(\undv)}\\
			&=
		y(\tau)\langle m,n\rangle_{\BS(\undv)}+ys(\tilde{\tau})\langle m',n'\rangle_{\BS(\undv)}\end{align*}
		and both $\langle m,n\rangle_{\BS(\undv)}$ and $\langle m',n'\rangle_{\BS(\undv)}$ vanish for $x\neq y$ by induction.
	\end{proof}
	
	\begin{definition}
		We call the pairing $\langle-,-\rangle_{\BS(\undw)}$ defined in \eqref{intform} the \emph{intersection form} of $\BS(\undw)$.
	\end{definition}
	
For any light leaf morphism  $\LL_{\undw,\unde}$, let $\flipLL_{\undw,\unde}\in \Hom^\bullet(\BS(\underline{x}),\BS(\underline{w}))$ be the adjoint morphism of $\LL_{\undw,\unde}$ with respect to the intersection form. If $\undw^\unde=\undw^\undf$ let $\bbLL_{\underline{w},\unde,\undf}=\flipLL_{\undw,\unde}\circ \LL_{\undw,f}$. We know from \cite[Theorem 5.5]{AbeBimodule} that the set $\{\bbLL_{\underline{w},\unde,\undf}\}_{\undw^\unde=\undw^\undf}$ is a basis of $\End^\bullet(\BS(\underline{w}))$ as a left $R$-module.	

If $x=\undw^\unde$, let $ll_{\underline{w},\unde}:=\flipLL_{\underline{w},e}(1^\otimes_{\undx})$. The set $\{ll_{\underline{w},\unde}\}_{\unde\subset \undw}$ is a basis of $\BS(\undw)$ as a left $R$-module (by \cite[Theorem 3.17]{AbeBimodule} or \cite[Theorem 12.25]{EMTWIntroduction}).
We have $\deg(ll_{\undw,\unde})=-\ell(\undw^\unde)+\defect(\unde)$. 
In particular, 
\[e=\can_{\undw^\unde}\iff\defect(\unde)=\ell(\undw)-\ell(\undw^\unde)\iff \deg(ll_{\undw,\unde})=\ell(\undw)-2\ell(\undw^\unde).\] 
If there is at least one $D$ in the decoration of $\unde$, then the inequality $\deg(ll_{\undw,\unde})\leq \ell(\undw)-2\ell(\undw^{\unde})-2$ holds.

Let $\{ll_{\undw,\unde}^*\}_{\unde\subset \undw}$ denote the dual basis of $\{ll_{\underline{w},\unde}\}_{\unde\subset \undw}$ with respect to the intersection form of $\BS(\undw)$.

\begin{remark}
	We can also think of duality diagrammatically (cf. \cite{EWSoergel}). In fact, taking the adjoint of a morphism with respect to the intersection form is equivalent to sending a morphism defined by a diagram $S$ to the morphism obtained by turning $S$ upside-down  \cite[Proposition A.12]{PatBases}.
\end{remark}

	\subsection{Support filtration}

	Since every $B\in \sbim$ is free as a left $R$-module, there is an inclusion $B \hookrightarrow B_Q=\bigoplus B_Q^x$.

	\begin{definition}
		We say that $b\in B$ is supported on a subset $A\subset W$ if $b\in \bigoplus_{x\in A}B_Q^x$. We denote by $\Gamma_A B$ the subset of elements supported on $A$. We write $\Gamma_{\geq y}B$ for $A=\{ x \in W \mid x \geq y\}$ and similarly for $\Gamma_{\leq y}B$. 
	\end{definition}
	
	\begin{remark}
		If the representation is reflection faithful the support filtration $\Gamma_{\leq y}$ defined above coincides with the one defined by Soergel in \cite[Definition 5.4]{SoeKazhdana} in terms of support of coherent sheaves (see \cite[Lemma 4.7]{PatSingular} or \cite[Prop. 3.25]{EKLPReduced}).
	\end{remark}
	
	\begin{lemma}\label{llbasis1}
		Let $x\leq \undw$. Then $\{ll_{\undw,\unde}\}_{\undw^\unde\leq x}$ is a basis of $\Gamma_{\leq x}\BS(\undw)$ as a left $R$-module and $\{ll_{\undw,\unde}^*\}_{\undw^\unde\geq x}$ is a basis of $ \Gamma_{\geq x} \BS(\undw)$ as a left $R$-module.
	\end{lemma}
	\begin{proof}
		The first statement follows from \cite[Theorem 3.17]{AbeBimodule}.
		From the same result, it also follows that $\{ll_{\undw,\unde}\}_{\undw^\unde\not \geq x}$ is a basis of $\Gamma_{\not \geq x}(\BS(\undw))$.
		
		By \Cref{nondeg}, the restriction of the intersection form to $\BS(\undw)_Q^x$ is non-degenerate for any $x\leq \undw$. Therefore, the orthogonal of $\Gamma_{\not \geq x}(\BS(\undw))_Q$ is $\Gamma_{ \geq x}(\BS(\undw))_Q$.
		
		The set $\{ll^*_{\undw,\unde}\}_{\unde\subset \undw}$ is an $R$-basis of $\BS(\undw)$ and $\{ll^*_{\undw,\unde}\}_{\undw^{\unde}\geq x}$ is a $Q$-basis of $\Gamma_{ \geq x}(\BS(\undw))_Q$. It remains to show that $\{ll^*_{\undw,\unde}\}_{\undw^{\unde}\geq x}$ also generates $\Gamma_{ \geq x}(\BS(\undw))$ over $R$.
		
		This follows because any element in $\gamma\in \Gamma_{ \geq x}(\BS(\undw))$ can be written as an $R$-linear combination $\gamma=\sum_{\unde} p_{\undw,\unde} ll^*_{\undw,\unde}$. But $\gamma$ can belong to $\Gamma_{ \geq x}(\BS(\undw))\subset \Gamma_{ \geq x}(\BS(\undw))_Q$ only if all the coefficients $p_{\undw,\unde}$ vanish whenever $\undw^\unde\not\geq x$.
	\end{proof}

	\begin{cor}\label{degll*}
		Let $b \in \Gamma_{\geq y}\BS(\undw)$. Then $\deg(b)\geq 2\ell(y)-\ell(\undw)$. 
	\end{cor}
	\begin{proof}
		This follows from \Cref{llbasis1} because $\deg(ll^*_{\undw,\unde})=-\deg(ll_{\undw,\unde})\geq 2\ell(\undw^{\unde})-
		\ell(\undw)$. 
	\end{proof}

\subsection{Cohomology}

\begin{definition}
	We denote by $D_{\undw}$ the left $R$-submodule of $\BS(\undw)$ spanned by the non-canonical light leaves. We call $D_{\undw}$ the \emph{defective submodule} of $\BS(\undw)$.
\end{definition}

 \begin{definition}
 	We define the \emph{cohomology submodule} $H_{\undw}\cug \BS(\undw)$ to be the orthogonal complement of $D_{\undw}$ with respect to the intersection form $\langle-,-\rangle_{\BS(\undw)}$.\label{s:tildeH}
 \end{definition}

 We have by definition $D_{\undw}= \bigoplus_{\unde\text{ non canonical}} R ll_{\undw,\unde}$ and $H_{\undw}=\bigoplus_{x\leq \undw} R ll^*_{\undw,\can_x}$.

\begin{lemma}\label{Pwelldef}
	The element $ll_{\undw,\can_x}^*\in \BS(\undw)$ does not depend on the choices made in the construction of the light leaves basis.
\end{lemma}
\begin{proof}
	Suppose we have fixed light leaf morphisms $\LL_{\undw,\unde}$ for any sequence $\unde\subset \undw$ and let $\tilLL_{\undw,\can_x}$ be another light leaf for $\can_x$ constructed using different choices of reduced expressions. Since braid moves fix $1^\otimes_{\undw}$, we can assume that both $\LL_{\undw,\can_x}$ and $\tilLL_{\undw,\can_x}$ are morphisms from $\BS(\undw)$ to the same bimodule $\BS(\undx)$, for some reduced expression $\undx$ of $x$.
	
	Because the set of double light leaves forms a basis of $\Hom^\bullet(\BS(\undw),\BS(\undx))$ we can write
	\begin{equation}\label{tilLL} \tilLL_{\undw,\can_x}= d \LL_{\undw,\can_x} + \sum_{\substack{\undw^\unde=x\\\unde\neq \can_x}} p_\unde \LL_{\undw,\unde}+\sum_{\undw^\unde=\undx^\undf<x} q_{\unde,\undf} \left(\flipLL_{\undx,\undf}\circ\LL_{\undw,\unde}\right)
\end{equation}
with $p_{\unde},q_{\unde,\undf},d \in R$. We can assume that \eqref{tilLL} is homogeneous and, for degree reasons, we must have $d\in \Bbbk$. Evaluating \eqref{tilLL} in $1^\otimes_{\undw}$, by \Cref{D=0}, we have
\[ 1^\otimes_{\undx} = d \cdot 1^\otimes_{\undx} +\sum_{\undx^\undf=y<x} q_{\can_y,\undf}\cdot ll_{\undx,\undf}.
\]
Since $\deg(ll_{\undx,\undf})>\deg(1^\otimes_{\undx})$ if $\undf \neq \can_x$, we conclude that $d=1$. Then, dualizing \eqref{tilLL} and evaluating it in $1^\otimes_{\undw}$, letting $\tilde{\flipLL}_{\undw,\unde}$ be adjoint of $\LL_{\undw,\unde}$ and $\tilll_{\undw,\can_x}=\tilde{\flipLL}_{\undw,\unde}(1^\otimes_\undx)$, we obtain
\begin{equation}\label{ll=ll}\tilll_{\undw,\can_x}=ll_{\undw,\can_x}+\sum_{\substack{\undw^\unde=x\\\unde\neq \can_x}} p_{\unde} ll_{\undw,\unde} + \sum_{\undw^\unde=y<x} q_{\unde,\can_y} ll_{\undw,\unde}.
\end{equation}
The subspaces $\Gamma_{<x}\BS(\undw)= \sum_{\undw^\unde=y< x} R\cdot ll_{\undw,\unde}$ do not depend on the specific choices made in the light leaves construction. Also the space $ \sum_{\undw^\unde=y<x} R\cdot ll_{\undw,\unde}+\Gamma_{<x} \BS(\undw)$ does not depend on the specific choices for degree reasons. Hence, if $\tilde{ll}^*_{\undw,\can_x}$ is the dual element of a basis of light leaves containing $\tilll_{\undw,\can_x}$, we have by \eqref{ll=ll} that
\[ \langle\tilde{ll}^*_{\undw,\can_x},ll_{\undw,\can_x}\rangle=\langle \tilll^*_{\undw,\can_x},\tilll_{\undw,\can_x}\rangle=1.\]
and $\langle\tilde{ll}^*_{\undw,\can_x},ll_{\undw,\unde}\rangle=0$ for any $\unde\neq \can_x$ with $\undw^\unde\not > x$.

Finally, if $\undw^\unde>x$, then $\deg(ll_{\undw,\unde})<\deg(ll_{\undw,\can_x})$ and $\deg\langle \tilde{ll}^*_{\undw,\can_x}, ll_{\undw,\unde}\rangle=-\ell(\undw) -\deg(ll_{\undw,\can_x})+\deg(ll_{\undw,\unde})<-\ell(\undw)$, so we have $\langle \tilde{ll}^*_{\undw,\can_x}, ll_{\undw,\unde}\rangle=0$ for degree reasons. We conclude that $\tilde{ll}^*_{\undw,\can_x}=ll^*_{\undw,\can_x}$.
\end{proof}

\begin{definition}
For $x\leq \undw$,	let $\calP_{\undw,x}:=ll_{\undw,can_x}^*$. The set $\{\calP_{\undw,x}\}$ is a basis of $H_{\undw}$ as a left $R$-module. 
\end{definition}

Because $ll_{\undw,\can_{\eid}}=c_{top}$, it is easy to check that $\calP_{\undw,\eid}=ll^*_{\undw,\can_{\eid}}=1^\otimes_{\undw}$.
We now prove a Pieri formula for $\calP_{\undw,x}$.

%

\begin{lemma}
	For any $x\leq \undw$ and $\lambda\in V$ we have
\begin{equation}\label{Pieri}
	\calP_{\undw,x}\cdot \lambda=x(\lambda)\calP_{\undw,x}+\sum_{\stackrel{x\xto{t}y}{y\leq w}}\partial_t(\lambda)\calP_{\undw,y}.
\end{equation}
\end{lemma}
\begin{proof}
	Since we have shown in \Cref{Pwelldef} that $\calP_{\undw,x}$ does not depend on the choices involved in the construction of the light leaves basis, we may choose $\LL_{\undw,\can_x}$ to be the light leaf morphism constructed using only trivial braid moves.
	
	Let $\undw=s_1\ldots s_l$ and let $(\eps_1,\ldots,\eps_l):=\can_x\in \{0,1\}^l$. Then $\LL_{\undw,\can_x}$ is a morphism from $\BS(\undw)$ to $\BS(\undx)$,
	where $\undx=t_1t_2\ldots t_k$ is the reduced expression of $x$ obtained by removing from $\undw$ all the $s_i$ such that the $\eps_i=0$.
	In particular, we have 
	\[\LL_{\undw,\can_x}= \psi_{s_1}^{\eps_1}\otimes \cdots \otimes \psi_{s_l}^{\eps_l},\]
	where 
	\[ \psi_{s_i}^{\eps_i}:=\begin{cases} \Iden_{B_{s_i}}&\text{if }\can_x(i)=1\\\ m_{s_i}&\text{if }\can_x(i)=0\end{cases}\]
	and $m_{s_i}:B_{s_i}\to R$ is the morphism defined by $f\otimes g\mapsto fg$. For $\unde\in \{0,1\}^l$ let $-\unde$ be the sequence obtained by inverting its $0$'s and $1$'s. It follows that $ll_{\undw,\can_{x}}=c_{-\can_{x}}$. 
	
	We recall the nil-Hecke relation in $B_s$. For any $f\in R$ we have
	
	\begin{equation}\label{nilheckerel} c_{id} f=s(f)c_{id}-\partial_s(f) c_s\in B_s\end{equation}

	For $1\leq i \leq \ell(x)$ let $x_i=t_{i} t_{i+1}\ldots t_k$ and $x_{\ell(x)+1}=\eid$. For $\unde\in \{0,1\}^l$, we denote by $\ehat$ the sequence obtained by replacing the $i$-th occurrence of $1$ in $\unde$ with a $0$.

	Let $\lambda\in V$. 
	Using the nil-Hecke relation \eqref{nilheckerel} repeatedly, we get
	$$ ll_{\undw,\can_x}\cdot \lambda=c_{-\can_x}\cdot \lambda=x(\lambda)c_{-\can_x}+\sum_{i=1}^{\ell(x)}\partial_{t_i}(x_{i+1}(\lambda))c_{-\can_x(\hat{\imath})}$$
	
	If $\can_x(\hat{\imath})$ is canonical, i.e., if it is decorated only with $U$'s, then $c_{-\can_x(\hat{\imath})}=ll_{\undw,\can_y}$ for some $y<x$ such that $y\xto{t} x$ with $t=x_{i+1}^{-1}t_ix_{i+1}\in T$ and $\partial_{t_i}(x_{i+1}(\lambda))=\partial_t(\lambda)$.

	If $\can_x(\hat{\imath})$ is not canonical, then $c_{-\can_x(\hat{\imath})}$ is in the image of a morphism $\BS(\undw^{\can_x(\hat{\imath})})\to\BS(\undw)$ and, as such, we have $c_{-\can_x(\hat{\imath})}\in \Gamma_{\leq \undw^{\can_x(\hat{\imath})}}\BS(\undw)$, with $\ell(\undw^{\can_x(\hat{\imath})})\leq \ell(x)-2$. Thus we can write
	\begin{equation}\label{canPieri}ll_{\undw,\can_x}\cdot \lambda =x(\lambda) ll_{\undw,\can_x}+\sum _{y\xto{t} x} \partial_t(\lambda)ll_{\undw,\can_y} + \theta,
\end{equation}	
	with $\theta\in \Gamma_{\{y \mid \ell(y)\leq \ell(x) -2\}}\BS(\undw)$. Furthermore, by \Cref{llbasis1} we have $\theta= \sum_j h_j ll_{\undw,\underline{f_j}}$ with $h_j\in R$ and $\ell(\undw^{\underline{f_j}})\leq \ell(x)-2$. 
	The element $\theta$ is homogeneous of $\deg(\theta)=\deg(ll_{\undw,\undx})+2$. 
	This forces the degree of $ll_{\undw,\underline{f_j}}$ to be too small for $\underline{f_j}$ to be canonical, in fact we have
	$$\deg ll_{\undx,f_j}\leq \deg ll_{\undw,\can_x}+2=\ell(w)-2\ell(x)+2<\ell(w)-2\ell(\undw^{f_j}),$$
	whence $\theta\in D_{\undw}$.
	
	By duality, using $\langle \calP_{\undw,x},\theta\rangle= 0$ for any $z$,	
	from \eqref{canPieri} we obtain the desired Pieri formula for the multiplication in the basis $\{\calP_{\undw,x}\}$ of $H_{\undw}$.
\end{proof}

Recall that by \Cref{equivZSBim}, there is a natural $Z$-action on $\BS(\undw)$ extending the structure of an $R$-bimodule.

\begin{prop}\label{Z1=H}
	Let $\pi:Z\to \BS(\undw)$ be the morphism of $Z$-modules defined by $\pi(z)=z\cdot 1^\otimes$. Then $\pi$ sends $\calP_x$ to $\calP_{\undw,x}$. In particular, the image of $\pi$ is the submodule $H_\undw$.
\end{prop}
\begin{proof}
	We have $\pi(\calP_x)\in \Gamma_{\geq x} \BS(\undw)$ and for degree reasons (cf. \Cref{degll*}) we have that $\pi(\calP_x)$ is a scalar multiple of $ll^*_{\undw,\can_x}$. We show that $\pi(\calP_x)= ll^*_{\undw,\can_x}$ by induction on $\ell(x)$.
	In fact, the claim is clear if $\ell(x)=0$ since they both coincide with $1^\otimes_{\undw}$. 
	
	Assume we know the claim for all $z<x$. Let $s\in S$ such that $xs<x$ and let $y=xs$, so that $y\xto{s} x$.	Then by \eqref{PieriZ} and \eqref{Pieri}
	We have for any $\lambda\in V$ that
	\begin{equation}\label{P=P} 0=(\pi(\calP_y)-\calP_{\undw,y})\cdot \lambda - y(\lambda)\cdot (\pi(\calP_y)-\calP_{\undw,y})=\sum_{y\xto{t} z}\partial_t(\lambda)(\pi(\calP_z)-\calP_{\undw,z}). \end{equation}
	Now, by linear independence all terms in the RHS of \eqref{P=P} must vanish. By Demazure surjectivity, we can always find $\lambda \in V$ such that $\partial_s(\lambda)\neq 0$ and we conclude that $\pi(\calP_x)=\calP_{\undw,x}$.
\end{proof}

For $w\in W$ let $Z_w=Z/I_w$, where $I_w$ is the ideal of $Z$ generated by $\{\calP_y\mid y\not \leq x\}$. 
\begin{prop}
	The quotient $Z_w$ is free as an $R$-module with basis given by the projection of $\{\calP_z\}_{z\leq w}$.
	
	Assume that $\undw$ is a reduced expression for $w$. Then $H_\undw$ is a $Z$-submodule contained in the indecomposable summand $B_w$ of $\BS(\undw)$ and is isomorphic to $Z_w$
\end{prop}

\begin{proof}
	From \cite[Theorem 5.7]{RZNil}, in the dual nil-Hecke ring $\Lambda$ we have $\xi^u\cdot \xi^v= \sum p_{u,v}^z\xi^z$ with $p_{u,v}^z=0$ unless $u\leq z$. So the submodule $\bigoplus_{x\not \leq w} R \xi^x$ is an ideal of $\Lambda$. By \Cref{equivZSBim}, the same is true for the ideal $\bigoplus_{x\not \leq w} R \calP^x$ of $Z$, which therefore coincides with $I_w$. It is now clear that $\bigoplus_{x\leq w} R \calP^x$ maps isomorphically to $Z_w$, and $\{\calP^x\}_{x\leq w}$ descends to a basis of $Z_w$. 
	
	The image $\pi(Z)$ is contained in $B_w$ because $B_w$ is a $Z$-module which contains $1^\otimes_\undw$ by \Cref{sct}. If $x\not \leq w$, then $\pi(\calP_x)\in \Gamma_{\geq x}B_w=0$, so $\pi$ factors through a map $\pi:Z_w\to B_w$, which is injective because $\{\calP_{\undw,x}\}_{x\leq \undw}$ is linearly independent over $R$.
\end{proof}


\begin{remark}\label{ecWeyl}
	Assume that $W$ is the Weyl group of a reductive group $G$. Let $B\subset G$ be a Borel subgroup and let $T\subset B$ be a maximal torus. For $w\in W$ let $X_w$ denote the corrisponding Schubert variety in $X:=G/B$. Then $B_w\cong \Coh^\bullet_T(X_w,\Bbbk)$ and the element $\calP_x$ can be described geometrically as the fundamental class of the Schubert variety. 
	Moreover, in this case have an isomorphism $Z\cong H^\bullet_T(X,\Bbbk)$ and the  basis $\{\calP_x\}$ corresponds to the basis given by the fundamental classes of Schubert cycles in cohomology.
\end{remark}

\begin{remark}
Similar arguments can be developed for one-sided singular Soergel bimodules. Let \( I \subset S \) and denote by \( W_I \) the subgroup generated by \( I \), and assume that \( W_I \) is finite. Under certain stricter assumptions on the realization (see \cite{WilSingular,AbeSingular}), for each \( w\in W/W_I \) one may define an indecomposable Soergel bimodule \( B_w^I \) as the unique submodule that appears as the direct summand containing \(1^\otimes_{\underline{w}}\) in any decomposition of the Bott--Samelson module $\BS(\undw)$ into indecomposable $(R,R^{W_I})$-submodules, where $\underline{w}$ is a reduced expression for the shortest element in the coset $W/W_I$. Here $R^{W_I}$ and $Z^{W_I}$ denote the $W_I$-invariants under the $W_I$-action.

Analogous to the ordinary case, it can be shown that \( B_w^I \) naturally carries the structure of a $Z^{W_I}$-module that extends its $(R,R^{W_I})$-bimodule structure. So one can define $H_w^I$ to be the cyclic $Z^{W_I}$-submodule generated by $1^\otimes_{\underline{w}}$. Using singular light leaves (see \cite{EKLPSingular}), it can be further shown that this module possesses a distinguished basis which has a similar description in terms of singular light leaves.
\end{remark}

		\section{Translation functors on \texorpdfstring{$\bar{Z}$}{Z}-mod}
		
		For $s\in S$, consider the subring $Z^s\cug Z$. It has a basis $\{\calP_v\}_{vs>v}$ as an $R$-module (cf. \Cref{Zsbasis}).
		Moreover, the ring $Z$ is a free $Z^s$-module with basis $\{1,\tau_s\}$, where $\tau_s:=(w(\varpi_s))_{w\in W}$.

		Let $R_+Z$ be right ideal of $Z$ generated by $R_+$, that is $R_+Z=\sum_{x\in W} R_+\calP_x $. 
		We define 
		\begin{equation}\label{s:barZ}
			\bar{Z}=Z/R_+Z=R/R_+\otimes_R Z\cong \Bbbk\otimes_R Z.
		\end{equation}
		Let $P_x:=1\otimes\calP_x\in \bar{Z}$.
		Then $\{P_x\}_{x\in W}$ is a basis of $\bar{Z}$ over $\Bbbk$.
		
		Let $\bar{Z}^s:=\Bbbk \otimes_R Z^s$. 
		Then $\bar{Z}$ is also a free $\bar{Z}^s$-module with basis $\{1,\bar{\tau_s}\}$, where $\bar{\tau_s}:=1\otimes \tau_s\in \bar{Z}$.

		\begin{prop}[{cf. \cite[Proposition 5.2]{FieCombinatorics}}]\label{adjoint}
 The two functors $\bar{Z}^s$-mod$\raw \bar{Z}$-mod defined by
				$$M\mapsto M\otimes_{\bar{Z}^s}\bar{Z}(2)\quad\text{ and }\quad M\mapsto\Hom^{\bullet}_{\bar{Z}^s}(\bar{Z},M)$$ are equivalent.
				
	The functor $\bar{Z}$-mod$\raw \bar{Z}$-mod given by 
				$$M\mapsto M\otimes_{ \bar{Z}^s} \bar{Z}[1]$$
				is self-adjoint.
		\end{prop}
		
		\begin{proof}
			Let $\{1^*,\bar{\tau_s}^*\}$ be the basis of $\Hom^{\bullet}_{\bar{Z}^s}(\bar{Z},\bar{Z}^s)$ dual to $\{1,\bar{\tau_s}\}$. Since $\deg 1=\deg 1^*=0$ and $\deg \bar{\tau_s}^*=-\deg \bar{\tau_s}=-2$ we have that the map of $\bar{Z}^s$-modules 
			$\Psi:\bar{Z}(2)\raw \Hom^{\bullet}_{\bar{Z}^s}(\bar{Z},\bar{Z}^s)$ defined by $1\mapsto \bar{\tau_s}^*$ and $\bar{\tau_s}\mapsto 1^*$ is an isomorphism.
			Because $\bar{Z}$ is free as a $\bar{Z}^s$-module, for any $\bar{Z}^s$-module $M$ we have a natural isomorphism of $\bar{Z}$-modules:
			\begin{center} 
				\begin{tikzpicture}
					\matrix(m)[matrix of math nodes, column sep=25pt, row sep=10pt]
					{\Hom^{\bullet}_{\bar{Z}^s}(\bar{Z},M) & M\otimes_{\bar{Z}^s}\Hom^{\bullet}_{\bar{Z}^s}(\bar{Z},\bar{Z}^s) & M\otimes_{\bar{Z}^s}\bar{Z}(2)\\
						\phi \quad & \quad \phi(1)\otimes 1^*+\phi(\bar{\tau_s})\otimes \bar{\tau_s}^*& \phi(1)\otimes \bar{\tau_s}+\phi(\bar{\tau_s})\otimes 1.\\};
					\path[-stealth] (m-1-1) edge node[above]{$\sim$} (m-1-2)
					(m-1-2) edge node[above] {$\Psi^{-1}$}(m-1-3);
					\path[|->]	(m-2-1) edge (m-2-2)
					(m-2-2) edge (m-2-3);
				\end{tikzpicture}
			\end{center}

			The second statement now follows since the restriction functor $\bar{Z}$-mod$\raw \bar{Z}^s$-mod is right adjoint to $-\otimes_{\bar{Z}^s}\bar{Z}$ and left adjoint to $\Hom^{\bullet}_{\bar{Z}^s}(\bar{Z},-)$. 
		\end{proof}
		
				For a Soergel bimodule $B$ we define $\bar{B}=\Bbbk\otimes_R B$. This is in a natural way a graded right $R$-module. All graded right $R$-modules arising this way are called \emph{Soergel modules}.
		Since $B$ is a $Z$-module, then $\bar{B}$ is also a a module over $\bar{Z}=Z/R_+Z$. 
		
		The following proof is based on unpublished notes by Soergel, in which he considers the case of finite Coxeter groups (Soergel's proof also appears in \cite[Proposition 1.10]{RicLa}).
		
		\begin{theorem}[Hom formula for Soergel modules]\label{homformula}
			Let $B,B'\in \sbim$. Then we have an isomorphism of graded right $R$-modules
			$$\Theta:\Bbbk\otimes_R\Hom^\bullet_{R\otimes R}(B',B)\xrightarrow{\sim}\Hom^\bullet_{\bar{Z}}(\Bbbk\otimes_R B',\Bbbk \otimes_R B).$$ 
			defined by $\Theta(z\otimes \phi)(z'\otimes b)=zz'\otimes \phi(b)$ for $z,z'\in \Bbbk$, $\phi\in \Hom^\bullet_{R\otimes R}(B',B)$ and $b\in B'$.
		\end{theorem}
		\begin{proof}
		Let $\phi:B'\to B$ be a morphism in $\sbim$. By \Cref{equivZSBim}, it is also a morphism of $Z$-modules, hence the resulting map $\Theta(z\otimes \phi)$ is a map of $\bar{Z}$-modules for any $z\in \Bbbk$.
			
			Because every indecomposable bimodule is a direct summand of a Bott--Samelson bimodule (cf. \Cref{sct}), it is enough to show the theorem for $B,B'$ being Bott--Samelson bimodules. Moreover, by adjunction (Proposition \ref{adjoint} and \cite[Proposition 5.10]{SoeKazhdana}) we can restrict ourselves to the case $B'=R$, that is to show
			$$\Bbbk\otimes_R\Hom^\bullet_{R\otimes R}(R,B)\cong\Hom^\bullet_{\bar{Z}}( \Bbbk,\Bbbk\otimes B).$$
			
			By sending $\phi:R\raw B$ to $\phi(1)$ we get $$\Hom^\bullet_{R\otimes R}(R,B)=\Gamma_{\eid}B.$$ 
			On the other hand, similarly, we obtain
			$$\Hom^\bullet_{\bar{Z}}( \Bbbk,\Bbbk\otimes B)\cong\left\{b\in \Bbbk\otimes_R B\mid P_x\cdot b=0\text{ for all }x\in W\setminus\{\eid\}\right\}=\bigcap_{x\neq \eid} Ann\left(P_x\right).$$ 
			The resulting map $\Bbbk\otimes_R \Gamma_{\eid} B\raw \bigcap_{x\neq \eid} Ann(P_x)\cug \Bbbk\otimes_R B$ is induced by the inclusion $\Gamma_{\eid}B\hookrightarrow B$.
			
			By \cite[Corollary 3.13]{AbeBimodule}, for any $x\in W$ the module $\Gamma_xB$ is free as a left $R$-module. In particular, $\Bbbk\otimes_R \Gamma_{\eid}B\cug \Bbbk\otimes_R B$ and thus $\Theta$ is injective.

			To show that $\Theta$ is also surjective, it is sufficient to show that if $b\in \Bbbk\otimes_R B$ and $b\not \in\Bbbk\otimes_R \Gamma_\eid B$, then there exists $x\in W\setminus \{0\}$ such that $P_x \cdot b\neq 0$.

			Fix an enumeration $w_1,w_2,w_3\ldots$ of the elements of $W$ which refines the Bruhat order.
			Let $\Gamma_{\geq h}B$ be the subset of elements supported in $\{w_i \mid i \leq h\}$.
			Let $h\in \bbN$ be such that $b\in \Bbbk\otimes_R\Gamma_{\leq h}B$ and $b\not\in\Bbbk\otimes_R\Gamma_{\leq h-1}B$. Let $x=w_h$.
			
			Multiplication by $\calP_x$ induces an isomorphism of $R$-bimodules 
			$$\calP_x\cdot(-):\Gamma_{\leq h}B/\Gamma_{\leq h-1}B\xra{\sim}\Gamma_xB.$$
			In fact, by \cite[Corollary 3.18]{AbeBimodule} we have $\Gamma_{\leq h}B/\Gamma_{< h}B\cong \Gamma_{\leq x}B/\Gamma_{<x}B$ and on $\Gamma_{\leq x}$ multiplying by $\calP_x$ is the same as multiplying on the left by $(\calP_x)_x=d_{x,x}=p_x$, hence its image is $p_x\left(\Gamma_{\leq x}B/\Gamma_{<x}B\right)=\Gamma_xB$ by \cite[Proposition 3.19]{AbeBimodule}. 
			As a consequence we obtain an isomorphism of right $R$-modules
			$$P_x\cdot(-):(\Bbbk\otimes_R \Gamma_{\leq h}B)/(\Bbbk\otimes_R\Gamma_{\leq h-1}B)\xra{\sim} \Bbbk\otimes_R\Gamma_xB\hookrightarrow \Bbbk \otimes_R B.$$
	In particular, $P_x\cdot b$ if $b\in \left(\Bbbk\otimes_R\Gamma_{\leq h}B\right)\setminus \left(\Bbbk\otimes_R\Gamma_{\leq h-1}B\right)$, proving the claim.
		\end{proof}
		
%

		\begin{cor}
			If $B$ is an indecomposable Soergel bimodule, then $\bar{B}=\Bbbk \otimes B$ is indecomposable as a $\bar{Z}$-module.
		\end{cor}
		
		From \Cref{sct}, we derive also a formula for the graded dimension of the space of morphisms between Soergel modules:
		\begin{equation}\label{Zhf}
			\grdim \Hom_{\bar{Z}}^\bullet(\Bbbk\otimes B, \Bbbk\otimes B')=(\bar{[B]},[B'])).
		\end{equation}

		\begin{remark}\label{finhomformula}
			Assume $\Bbbk=\bbR$. If $W$ is a finite Coxeter group, then the ring $Z\cong R\otimes_{R^W}R$ (cf. \cite[Theorem 4.3]{FieCombinatorics} and \cite[Lemma 4.3.1]{WilSingular}). Hence $\bar{Z}\cong \Bbbk\otimes_{R^W} R\cong R/R^W_+$ is the coinvariant ring. In particular, $\bar{Z}$ is generated in degree $2$ and the map $R\raw \bar{Z}$ is surjective. Clearly, in this case we can replace $\bar{Z}$ by $R$ (acting on the right) in the statement of Theorem \ref{homformula} and in \eqref{Zhf}.
		\end{remark}

		\section{Counterexamples}\label{counter}
		
		In general, for an infinite Coxeter group it is false that 
		\begin{equation}\label{naivehom}
			\Bbbk\otimes_R\Hom_{R\otimes R}(B,B') \cong \Hom_R(\Bbbk\otimes B, \Bbbk\otimes B').
		\end{equation}
		
		We now discuss two examples where \eqref{naivehom} fails. Furthermore, in the first example, we illustrate an example of an indecomposable Soergel bimodule $B$ such that $\Bbbk\otimes B$ is not indecomposable as a right $R$-module.
		
		\subsection{A counterexample in the affine Weyl group of type \texorpdfstring{$\tilde{A}_2$}{A??}}\label{cex1} Let $\Bbbk=\bbR$. Let $\tilde{W}$ be an affine Weyl group and let $\mathfrak{h}$ be a realization for $\tilde{W}$ in the sense of Kac (as in \cite[Proposition 1.1(2)]{RicLa}). 
		
			Let $W\cug \tilde{W}$ be the corresponding finite Weyl group and $G$ be the corresponding simply-connected semisimple group associated to $W$. Let $\calK=\bbC((t))$ and $\calO=\bbC[[t]]$, and let $\graff:=G(\calK)/G(\calO)$
		denote the affine Grassmannian. Let $\pi_0:G(\calO)\raw G$ be the map defined by sending $t\mapsto 0$, and let $B\subset G$ be a Borel subgroup of $G$. The group $I=\pi_0^{-1}(B)$ is called the \emph{Iwahori subgroup} of $G(\calK)$. The quotient 
		$\flaff=G(\calK)/I$
		is called the \emph{affine flag variety} of $G$. Let $p:\flaff \to \graff$ denote the projection map.
				\begin{prop}\label{trivialproj}
			The fiber bundle $p:\flaff \raw \graff$ is topologically trivial, i.e., $\flaff\cong \graff\times G/B$ as topological spaces.
		\end{prop}
		\begin{proof}
			We sketch the proof here and refer to \cite[\S 2.3]{PatHodge} for more details
			
			Let $B\subset G$ be a Borel subgroup of $G$ and $T\subset B$ be a maximal torus. Let $K$ be a maximal compact subgroup of $G$ and let $T_\bbR=T\cap K$. The Iwasawa decomposition implies that we have a homeomorphism $G/B\cong K/T_\bbR$.
			The space $G(\bbC[t,t^{-1}])$ is the space of algebraic maps $\bbC^*\raw G$. Let $L_{pol}K$ be the subspace of $G(\bbC[t,t^{-1}])$ consisting of maps that sends $S^1\cug \bbC^*$ into $K$. We have a subspace $\Omega_{pol}K\cug L_{pol}K$ of maps that send $1\in S^1$ to $1\in G$. 
			 Then the inclusion $\Omega_{pol}K\hookrightarrow G(\bbC[t,t^{-1}])$ induces a homeomorphism $\Omega_{pol}K\cong \graff$ by \cite[Theorem 8.6.3]{PSLoop}.
			Similarly, the affine flag variety $\flaff$ can be identified with $L_{pol}K/T_\bbR$, where $T_\bbR=T\cap K$ by \cite[Proposition 8.7.6]{PSLoop}.
			We have a $T_\bbR$-equivariant homeomorphism $\Omega_{pol} K\times K\cong L_{pol}K$. By taking the quotients of both sides by $T_\bbR$ we obtain a homeomorphism $\Omega_{pol} K\times K/T_\bbR\xra{\sim} L_{pol}K/T_{\bbR}$ defined by $(x,yT_\bbR)\mapsto xyT_\bbR$.
			This gives an isomorphism of fiber bundles over $\Omega_{pol} K$:
			\begin{center} 
				\begin{tikzpicture}
					\matrix(m)[matrix of math nodes, column sep=20pt, row sep=20pt]
					{\Omega_{pol} K\times K/T_\bbR & & L_{pol}K/T_\bbR\\
						& \Omega_{pol} K &\\};
					\path[-stealth] (m-1-1) edge node[above]{$\sim$} (m-1-3)
					edge (m-2-2)
					(m-1-3) edge node[below]{$p$} (m-2-2);
				\end{tikzpicture}
			\end{center}
			It follows that the projection $p$ is a topologically trivial fiber bundle.
		\end{proof}
		
		By the Künneth theorem, we obtain 
		\begin{equation}\label{deco1}H^\bullet(\flaff,\bbR)\cong H^\bullet(\graff,\bbR)\otimes_{\bbR} H^\bullet(G/B,\bbR).\end{equation}
		(see also \cite{LeeCombinatorial} for a more detailed description of this isomorphism).

		All cohomology and intersection cohomology groups below are taken with coefficients in $\bbR$.
		
		In \cite{HaeKazhdan} H\"arterich showed that for any $w\in \tilde{W}$ we have $\bbR\otimes B_w\cong \IH^\bullet(\flaff_w)$, where $\flaff_w=\bar{I\cdot wI/I}\cug \flaff$ is the corresponding Schubert variety and $\IH$ denote the intersection cohomology.

		Fix $\barw \in \tilde{W}/W$ and let $\graff_{\barw}\cug \graff$ be the corresponding Schubert variety. Let $w$ be the longest element in the coset $\barw$. Then we have $\flaff_{w}=p^{-1}(\graff_\barw)$. 
		Since $p$ is a topologically trivial fiber bundle, the same holds for the restriction $p:\flaff_{w}\raw \graff_\barw$. We have
		\begin{equation}\label{deco2}
			\IH^\bullet(\flaff_{w})\cong H^\bullet(G/B)[d]\otimes_\bbR \IH^\bullet(\graff_\barw)
		\end{equation}
		where $d=\dim_\bbR(G/B)$. The $\Coh^\bullet(\flaff)$-module structure on $\IH^\bullet(\flaff_{w})$ is given, in terms of the isomorphism \eqref{deco1} and \eqref{deco2}, by $(f\otimes f')(g\otimes g')=fg\otimes f'g'$.
		It follows that if $\IH^\bullet(\graff_\barw)$ decomposes as a $\Sym(H^2(\graff))$-module, then $\Coh^\bullet(\flaff_{w})$ decomposes as an $R$-module.
		
		Now assume further that the group $G$ is simple. Then $\Coh^\bullet(\graff)\cong Z^W$ and it follows $\Coh^2(\graff)$ is one-dimensional and it is generated by $P_u$, where $u$ is the unique simple reflection not in $W$. Therefore $\Sym(H^2(\graff))$ is isomorphic to the polynomial ring $\bbR[P_u]$, with $\deg(P_u)=2$.
		The primitive decomposition given by the hard Lefschetz theorem for $\IH^\bullet(\graff_\barw)$ is
		\begin{equation}\label{primdec}\IH^\bullet(\graff_\barw)= \bigoplus_{s\geq r \geq 0}(P_u)^r pIH^s(\graff_\barw)\end{equation}
		where \[pIH^s(\graff_\barw):= \Ker\left(P_u^{s+1}:IH^{-s}(\graff_\barw)\to IH^{s+2}(\graff_\barw)\right).\]

		 Note that there are very few Schubert varieties $\graff_\barw$ for which we have $\dim \IH^i(\graff_\barw)\leq 1$ for all $i$, and that if $\dim \IH^i(\graff_\barw)\geq 2$ for some $i$ then \eqref{primdec} shows that  $\IH^\bullet(\graff_\barw)$ cannot be indecomposable as a $\Sym(H^2(\graff))$-module.
		This describes how to produce many examples of indecomposable Soergel bimodules $B_w$ such that $\bar{B_w}$ is decomposable.
		
		The smallest explicit example is as follows:
		Let $\tilde{W}$ be the affine Weyl group of type $\tilde{A}_2$, that is, 
		$\tilde{W}:=\langle s,t,u\rangle$ and $m_{st}=m_{tu}=m_{us}=3$.
		Let $W$ be the subgroup generated by $s,t$ and let $w=stutst$, so tht $\bar{w}=stu$. Then $\bbR\otimes_R B_w=\IH^\bullet(\flaff_w)=H^\bullet(G/B)[3]\otimes_\bbR \IH^\bullet(\graff_{stu})$, where $G=SL_3(\bbR)$. 
		We have $\dim \IH^1(\graff_{stu})\geq \dim H^4(\graff_{stu})=2$, since $\Coh^4(\graff_{stu})$ is generated by $P_{su}$ and $P_{tu}$.
		Hence the Soergel module $\bbR\otimes_R B_w$ is decomposable as an $R$-module.

		\subsection{A counterexample in the universal Coxeter group of rank 3}\label{cex2} The following is another smaller counterexample to \eqref{naivehom} where we can see algebraically in more detail what happens.
		Let $W$ be the universal Coxeter group of rank $3$, i.e., $W=\langle s,t,u \rangle $ with $m_{st}=m_{tu}=m_{us}=\infty$. Let $\undw=stustu$ and consider the bimodule $\BS(\undw)$.

		For $\unde\in \{0,1\}^6$ let $c_{\unde}$ be the string basis element defined as in \eqref{c_e}. Consider the element
		$$b:=c_{000011}-c_{000101}+c_{000110}-c_{001010}+c_{001100}
		-c_{010001}-2c_{010010}+$$
		$$+c_{011000}-c_{010100}+c_{100001}-c_{100010}-c_{101000}+c_{110000}\in \BS(\undw).$$
		
		The element $b$ has degree $2$.
		The coefficient of $\bfH_{\eid}$ in the Kazhdan--Lusztig basis element $[B_w]$ is $v^6+v^4$. Hence, the submodule $\Gamma_\eid \BS(\undw)$ lies in degree $\geq 4$. 
		Then $b\not \in \Gamma_\eid(\BS(\undw))$ but it can be checked that the projection $\bar{b}\in \bar{\BS(\undw)}$ belongs to $Ann(R_+)$.
			It follows that the map $\bbR\raw \bbR\otimes_{R} \BS(\undw)$ defined by $1\mapsto \bar{b}$ is a map of right $R$-bimodules which does not arise from any bimodule map $R\raw \BS(\undw)$. Thus, the map $\Theta$ from \Cref{homformula} is not surjective onto $R$-module morphisms. 
		The correctness of this counterexample can be verified via a Sagemath worksheet \cite{sagemath} available at \url{https://lpatimo.github.io/Counterexample.ipynb}.

		\begin{remark}\label{lastremark}
			These two counterexamples discussed above allow us to negatively answer a question posed by Soergel in \cite[Remark 6.8]{SoeKazhdana}. In general, for infinite Coxeter groups, there may exist no non-zero function $c_y\in R\otimes R$ homogeneous of degree $2\ell(y)$ such that $c_y$ is supported on $Gr(\leq y)$ and vanishes on $Gr(<y)$. In fact, if such elements $c_y\in R\otimes_\Bbbk R$ existed, we could use them to play the role of $\calP_y$ in 
			the proof of Proposition \ref{homformula}, and this would imply the isomorphism \eqref{naivehom}. 
		\end{remark}

\def\cprime{$'$}

\Address

\begin{thebibliography}{EMTW20}
	
	\bibitem[Abe21]{AbeBimodule}
	Noriyuki Abe.
	\newblock A bimodule description of the {H}ecke category.
	\newblock {\em Compos. Math.}, 157(10):2133--2159, 2021.
	
	\bibitem[Abe24]{AbeHomomorphism}
	Noryiuki Abe.
	\newblock A {Homomorphism} between {B}ott--{S}amelson bimodules.
	\newblock {\em Nagoya Math. J.}, 256:761--784, 2024.
	
	\bibitem[Abe25]{AbeSingular}
	Noriyuki Abe.
	\newblock Singular {S}oergel {B}imodules for {R}ealizations.
	\newblock {\em Int. Math. Res. Not. IMRN}, (1):1--29, 2025.
	
	\bibitem[All13]{AllReflection}
	Daniel Allcock.
	\newblock Reflection centralizers in {C}oxeter groups.
	\newblock {\em Transform. Groups}, 18(3):599--613, 2013.
	
	\bibitem[Ara89]{AraCohomologie}
	Alberto Arabia.
	\newblock Cohomologie {$T$}-\'equivariante de la vari\'et\'e de drapeaux d'un
	groupe de {K}ac-{M}oody.
	\newblock {\em Bull. Soc. Math. France}, 117(2):129--165, 1989.
	
	\bibitem[BB05]{BBCombinatorics}
	Anders Bj\"orner and Francesco Brenti.
	\newblock {\em Combinatorics of {C}oxeter groups}, volume 231 of {\em Graduate
		Texts in Mathematics}.
	\newblock Springer, New York, 2005.
	
	\bibitem[BE09]{BEShape}
	Anders Bj{\"o}rner and Torsten Ekedahl.
	\newblock On the shape of {Bruhat} intervals.
	\newblock {\em Ann. Math. (2)}, 170(2):799--817, 2009.
	
	\bibitem[BMMN02]{BMMWRigidity}
	Noel Brady, Jonathan~P. McCammond, Bernhard M\"uhlherr, and Walter~D. Neumann.
	\newblock Rigidity of {C}oxeter groups and {A}rtin groups.
	\newblock In {\em Proceedings of the {C}onference on {G}eometric and
		{C}ombinatorial {G}roup {T}heory, {P}art {I} ({H}aifa, 2000)}, volume~94,
	pages 91--109, 2002.
	
	\bibitem[EKLP24]{EKLPSingular}
	Ben Elias, Hankyung Ko, Nicolas Libedinsky, and Leonardo Patimo.
	\newblock Singular light leaves, 2024.
	\newblock \href{https://arxiv.org/abs/2401.03053}{arXiv:2401.03053}.
	
	\bibitem[EKLP25]{EKLPReduced}
	Ben Elias, Hankyung Ko, Nicolas Libedinsky, and Leonardo Patimo.
	\newblock On reduced expressions for core double cosets.
	\newblock {\em Math. Res. Lett.}, to appear, 2025.
	\newblock \href{https://arxiv.org/abs/2402.08673}{arXiv:2402.08673}.
	
	\bibitem[Eli16]{EliTwo}
	Ben Elias.
	\newblock The two-color {S}oergel calculus.
	\newblock {\em Compos. Math.}, 152(2):327--398, 2016.
	
	\bibitem[EMTW20]{EMTWIntroduction}
	Ben Elias, Shotaro Makisumi, Ulrich Thiel, and Geordie Williamson.
	\newblock {\em Introduction to {S}oergel bimodules}, volume~5 of {\em RSME
		Springer Series}.
	\newblock Springer, Cham, 2020.
	
	\bibitem[EW14]{EWHodge}
	Ben Elias and Geordie Williamson.
	\newblock The {H}odge theory of {S}oergel bimodules.
	\newblock {\em Ann. of Math. (2)}, 180(3):1089--1136, 2014.
	
	\bibitem[EW16]{EWSoergel}
	Ben Elias and Geordie Williamson.
	\newblock Soergel calculus.
	\newblock {\em Represent. Theory}, 20:295--374, 2016.
	
	\bibitem[Fie08]{FieCombinatorics}
	Peter Fiebig.
	\newblock The combinatorics of {C}oxeter categories.
	\newblock {\em Trans. Amer. Math. Soc.}, 360(8):4211--4233, 2008.
	
	\bibitem[GKM98]{GKMEquivariant}
	Mark Goresky, Robert Kottwitz, and Robert MacPherson.
	\newblock Equivariant cohomology, {K}oszul duality, and the localization
	theorem.
	\newblock {\em Invent. Math.}, 131(1):25--83, 1998.
	
	\bibitem[H{\"a}r99]{HaeKazhdan}
	Martin H{\"a}rterich.
	\newblock {\em Kazhdan--{L}usztig-basen, unzerlegbar {B}imoduln und die
		{T}opologie der {F}ahnenmannigfaltigkeit einer {K}ac-{M}oody-{G}roup}.
	\newblock 1999.
	\newblock Thesis (Ph.D.)--Freiburg University
	\url{https://freidok.uni-freiburg.de/data/18}.
	
	\bibitem[KK86]{KKNil}
	Bertram Kostant and Shrawan Kumar.
	\newblock The nil {H}ecke ring and cohomology of {$G/P$} for a {K}ac-{M}oody
	group {$G$}.
	\newblock {\em Adv. in Math.}, 62(3):187--237, 1986.
	
	\bibitem[Lee19]{LeeCombinatorial}
	Seung~Jin Lee.
	\newblock Combinatorial description of the cohomology of the affine flag
	variety.
	\newblock {\em Trans. Amer. Math. Soc.}, 371(6):4029--4057, 2019.
	
	\bibitem[Lib08]{LibSur}
	Nicolas Libedinsky.
	\newblock Sur la cat\'egorie des bimodules de {S}oergel.
	\newblock {\em J. Algebra}, 320(7):2675--2694, 2008.
	
	\bibitem[Pat18a]{PatHodge}
	Leonardo Patimo.
	\newblock {\em Hodge theoretic aspects of {S}oergel bimodules and
		representation theory}.
	\newblock PhD thesis, University of Bonn, 2018.
	\newblock \url{http://hss.ulb.uni-bonn.de/2018/4955/4955.htm}.
	
	\bibitem[Pat18b]{PatNeron}
	Leonardo Patimo.
	\newblock The {N}\'{e}ron-{S}everi {L}ie algebra of a {S}oergel module.
	\newblock {\em Transform. Groups}, 23(4):1063--1089, 2018.
	
	\bibitem[Pat22a]{PatBases}
	Leonardo Patimo.
	\newblock Bases of the intersection cohomology of {G}rassmannian {S}chubert
	varieties.
	\newblock {\em J. Algebra}, 589:345--400, 2022.
	
	\bibitem[Pat22b]{PatSingular}
	Leonardo Patimo.
	\newblock Singular {Rouquier} complexes.
	\newblock {\em Proc. Lond. Math. Soc. (3)}, 125(6):1332--1352, 2022.
	
	\bibitem[PS86]{PSLoop}
	Andrew Pressley and Graeme Segal.
	\newblock {\em Loop groups}.
	\newblock Oxford Mathematical Monographs. The Clarendon Press, Oxford
	University Press, New York, 1986.
	\newblock Oxford Science Publications.
	
	\bibitem[Ric19]{RicLa}
	Simon Riche.
	\newblock La th\'{e}orie de {H}odge des bimodules de {S}oergel [d'apr\`es
	{S}oergel et {E}lias-{W}illiamson].
	\newblock Number 414, S\'{e}minaire Bourbaki. Vol. 2017/2018. Expos\'{e}s
	1136--1150, pages Exp. No. 1139, 125--166. 2019.
	
	\bibitem[RZ23]{RZNil}
	Edward Richmond and Kirill Zainoulline.
	\newblock Nil-hecke rings and the schubert calculus, 2023.
	
	\bibitem[{Sag}24]{sagemath}
	{Sage Developers}.
	\newblock {\em {S}ageMath, the {S}age {M}athematics {S}oftware {S}ystem
		({V}ersion 10.4)}, 2024.
	\newblock {\tt https://www.sagemath.org}.
	
	\bibitem[Soe07]{SoeKazhdana}
	Wolfgang Soergel.
	\newblock Kazhdan--{L}usztig-{P}olynome und unzerlegbare {B}imoduln \"uber
	{P}olynomringen.
	\newblock {\em J. Inst. Math. Jussieu}, 6(3):501--525, 2007.
	
	\bibitem[Wil11]{WilSingular}
	Geordie Williamson.
	\newblock Singular {S}oergel bimodules.
	\newblock {\em Int. Math. Res. Not. IMRN}, (20):4555--4632, 2011.
	
\end{thebibliography}
\end{document}